\documentclass[a4paper,11pt]{article}

\usepackage{amsmath,amssymb,amsthm,amscd,euscript, mathrsfs,dsfont,mathtools,bbm,ushort,url,xcolor,enumerate,graphicx,ccaption,stmaryrd}
\usepackage{relsize} 
\usepackage{hyperref}

\usepackage{Baskervaldx} 

\usepackage{hyperref}
\usepackage{xcolor}
\definecolor{unbleu}{rgb}{0.03, 0.15, 0.4}
\definecolor{monvert}{rgb}{0.0,.5,0.0}
\definecolor{britishracinggreen}{rgb}{0.0, 0.26, 0.15}
\definecolor{monbleu}{rgb}{0,.2,.8}
\definecolor{monautrebleu}{rgb}{0,0.4,.75}
\definecolor{applegreen}{rgb}{0.55, 0.71, 0.0}
\definecolor{monrouge}{rgb}{0.8, 0.0, 0.0} 
\definecolor{cadmiumgreen}{rgb}{0.0, 0.42, 0.24}
\definecolor{royalblue(traditional)}{rgb}{0.0, 0.14, 0.4}
\definecolor{black}{rgb}{0.0, 0.0, 0.0}

\hypersetup{
pdfborder = {0 0 0},
colorlinks,
linkcolor=unbleu,
citecolor=unbleu,
urlcolor=unbleu
}

\numberwithin{equation}{section}  

\newtheorem{theorem}{{\scshape {\bfseries Theorem}}}[section]
\newtheorem{lemma}[theorem]{{\scshape {\bfseries Lemma}}}
\newtheorem{proposition}[theorem]{{\scshape {\bfseries Proposition}}}
\newtheorem{corollary}[theorem]{{\scshape {\bfseries Corollary}}}
\newtheorem{remark}{{\scshape {\bfseries Remark}}}[section]

\renewenvironment{proof}[1]
{\noindent{{\bf{\small{P}{\scriptsize ROOF}}}.}\hspace{0.1cm} #1}
{$\;\blacksquare$\newline}

\newcommand{\integers}{\mathds{Z}_{{\scriptscriptstyle >0}}}
\newcommand{\entiers}{\mathds{Z}_{{\scriptscriptstyle \geq 0}}}
\newcommand{\real}{\mathds{R}}

\newcommand{\complex}{\mathds{C}}
\newcommand{\expectation}{{\mathds E}}
\newcommand{\Oun}{\mathcal{O}(1)} 
\newcommand{\proba}{{\mathbb P}}
\newcommand{\esperance}{{\mathds{E}}}
\newcommand{\ms}{m_{*}{\scriptstyle (K)}}

\newcommand{\dd}{\mathrm{d}}
\newcommand{\e}{\mathrm{e}}
\newcommand{\Dom}{\mathrm{Dom}}
\newcommand{\LKB}{L_{\sK}}
\newcommand{\LKl}{\mathcal{L}_{\sK}}
\newcommand{\LKlp}{\mathcal{L}_{\sK_p}}
\newcommand{\LKL}{\mathscr{L}_{\sK}}
\newcommand{\HO}{\EuScript{H}_{*}}
\newcommand{\PO}{\EuScript{M}_{0}}
\newcommand{\OU}{\EuScript{OU}_{*}}
\newcommand{\OUP}{\EuScript{Q}_{0}}
\newcommand{\un}{{\mathds{1}}}
\newcommand{\xf}{x_{*}}
\newcommand{\eK}[1]{e_{#1}^{{\scriptscriptstyle (K)}}}
\newcommand{\Ldeux}{L^{2}}
\newcommand{\ldeux}{\ell^{2}}
\newcommand{\TK}{T^{{\scriptscriptstyle (K)}}_0} 
\newcommand{\HK}{\mathscr{H}}
\newcommand{\nuK}{\nu^{{\scriptscriptstyle (K)}}}
\newcommand{\piK}{\pi^{{\scriptscriptstyle (K)}}}
\newcommand{\rhoK}[1]{\rho_{#1}^{{\scriptscriptstyle (K)}}}
\newcommand{\rhoKp}[2]{\rho_{#1}^{{\scriptscriptstyle(K_{#2})}}}

\newcommand{\phiKp}[2]{\phi_{#1}^{(K_{#2})}}

\newcommand{\phiK}[1]{{\phi^{{\scriptscriptstyle (K)}}#1}}

\newcommand{\spectre}{\mathrm{Sp}}
\newcommand{\ic}{\mathrm{i}}
\newcommand{\tlamz}{b'(0)}
\newcommand{\tmuz}{d'(0)}
\newcommand{\nge}{n_{\mathrm{g}}({\scriptstyle K},\eta)}
\newcommand{\tnge}{\tilde n_{\mathrm{g}}({\scriptstyle K},\eta)}
\newcommand{\nde}{n_{\mathrm{d}}({\scriptstyle K},\,\eta)}
\newcommand{\nl}{n_{\ell}{\scriptstyle (K)}}
\newcommand{\nr}{n_{r}{\scriptstyle (K)}}
\newcommand{\nmin}{n_{\mathrm{min}}{\scriptstyle (K)}}
\newcommand{\sK}{{\scriptscriptstyle K}}
\newcommand{\psK}{{\scriptscriptstyle (K)}}
\newcommand{\psKp}{{\scriptscriptstyle (K_p)}}
\newcommand{\Vn}{V_{n}(K)}
\newcommand{\lambdaK}{\lambda^{{\scriptscriptstyle (K)}}}
\newcommand{\muK}{\mu^{{\scriptscriptstyle (K)}}}
\newcommand{\Kxf}{K\xf}

\title{Large population limit of the spectrum of killed birth-and-death processes}

\author{J.-R. Chazottes$^{\textup{(a)}}$, P. Collet$^{\textup{(a)}}$, S. M\'el\'eard$^{\textup{(b)}}$\\
{\small $^{\textup{(a)}}$ CPHT, CNRS, Ecole polytechnique, IP Paris, Palaiseau, France}\\
{\small $^{\textup{(b)}}$ CMAP, CNRS, Ecole polytechnique, IP Paris, Palaiseau, France}
}

\begin{document} 

\maketitle

\begin{abstract}
We consider a general class of birth-and-death processes with state space $\{0,1,2,3,\ldots\}$ which describes the size of a population going eventually to extinction with probability one.
We obtain the complete spectrum of the generator of the process killed at $0$ in the large population limit, that is, we scale the process by a parameter $K$, and take the limit $K\to+\infty$.
We assume that the differential equation $\dd x/\dd t=b(x)-d(x)$  describing the infinite population limit  (in any finite-time interval) has a repulsive fixed point at $0$, and an attractive fixed point $\xf>0$.
We prove that, asymptotically, the spectrum is the superposition of two spectra. One is the spectrum of the generator of an Ornstein-Uhlenbeck process, which is
$n(b'(\xf)-d'(\xf))$, $n\ge 0$. The other one is the spectrum of a continuous-time binary branching process conditioned on non-extinction, and is given by $n(d'(0)-b'(0))$, $n\ge 1$.
A major difficulty is that different scales and function spaces are involved. We work at the level of the eigenfunctions that we split
over different regions, and study their asymptotic dependence on $K$ in each region. In particular, we prove that the spectral gap goes to
$\min\big\{b'(0)-d'(0),\,d'(\xf)-b'(\xf)\big\}$.
This work complements a previous work of ours in which we studied the approximation of the quasi-stationary distribution and of the mean time to extinction.

\smallskip

polynomials, quasi-eigenvectors, quasi-stationary 
distribution.
\end{abstract}

\newpage

\tableofcontents

\newpage

\section{Introduction, heuristics, and main result}

\subsection{The context}

We consider a class of birth-and-death processes $(X^{\sK}_t)_{t\geq 0}$ with state space $\entiers$\footnote{We denote by $\entiers$ the set of non-negative integers, and by $\integers$ the set of 
positive integers.} which describes how the size of a single population evolves according to birth and death rates of the form
\begin{equation}\label{def-lambdaKnmuKn}
\lambdaK_{n}=K\, b\left(\frac{n}{K}\right)\quad\text{and}\quad \muK_{n}=K\,d\left(\frac{n}{K}\right)
\end{equation}
where $n\geq 1$, and $K\in \integers$ is a scaling parameter, often called `carrying capacity'.  We suppose that $b(0)=d(0)=0$, implying that $0$ is an absorbing
state for the process, modelling extinction, and our assumptions are such that the probability to reach this state is equal to one.
The unique stationary distribution is the Dirac measure at $0$, so a relevant distribution to look for is a quasi-stationary distribution.
A probability measure $\nuK$ on the positive integers is a quasi-stationary distribution if, for all $t>0$ and for all subsets $A\subset \integers$, one has
$\proba_{\nuK}(X^{\sK}_t\in A \,|\, T^{{\scriptscriptstyle (K)}}_0>t)=\nuK(A)$, where $\TK$ is the extinction time, that is, the smallest $t>0$ such that
$X^{\psK}_t=0$.
In other words, a quasi-stationary distribution plays the role of a stationary distribution when conditioning upon non-extinction.
We refer to \cite{CMSM,MV} for more informations about quasi-stationary distributions.

When $K\to\infty$, the trajectories of the rescaled process $(K^{-1}X^{\psK}_t)_{t\geq 0}$ converge in probability, in any fixed time-window, to the solutions of the differential equation
\begin{equation}\label{the-edo}
\frac{\dd x}{\dd t} =b(x)-d(x)
\end{equation}
if the initial condition state is for instance of the form $\lfloor Kx_0\rfloor$ for a given $x_0>0$. 
We assume that  the functions $b$ and $d$ only vanish at $0$, and that
\[
d'(0) - b'(0)<0
\]
meaning that the fixed point $0$ is repulsive. We also assume that there is a unique attractive fixed point $\xf>0$, that is 
\[
b(\xf)=d(\xf)\quad\text{and}\quad b'(\xf)-d'(\xf)<0.
\]
(We will give the complete set of assumptions on the functions $b$ and $d$ later on.) 

 A famous example is the so-called  logistic process for which $b(x)=\lambda x$, $d(x)=x(\mu +x)$,  where $\lambda$ and $\mu$ are positive real numbers.
We assume that  $\lambda>\mu$ and we have $\xf=\lambda-\mu$.

In \cite{CCM1} we obtained the precise asymptotic behaviour of the first eigenvalue of the generator $\LKB$ of the process killed  at $0$, and also of the law of the 
extinction time starting from the quasi-stationary distribution (among other results).  
Here we go further and obtain the complete spectrum of the generator of the killed process, in the limit $K\to\infty$. 
In  particular, the knowledge of the spectral gap allows us to obtain the time of relaxation for the process conditioned on non-extinction to 
obey the quasi-stationary distribution.  

\subsection{Notations for basic function spaces}

We denote by
\begin{itemize}
\item
$\mathscr{D}$ the space of $C^\infty$ $\complex$-valued functions with compact support on $\real$;
\item
$c_{00}$ the space of $\complex$-valued sequences with finitely many nonzero values;
\item
$\ldeux$ the space of square-summable $\complex$-valued sequences equipped with the standard scalar product.
\item
$\Ldeux$ the space of square-integrable $\complex$-valued functions with respect to Lebesgue measure on $\real$.
\end{itemize}
We will define several operators on $c_{00}$ and will consider their closure on $\ldeux$. 
For simplicity, we will use the same notation for an operator and its closure. As we will see later, there is no ambiguity on the extensions.

\subsection{Heuristics}

The fundamental object in this paper is the spectrum of the following operator that we momentarily define on $c_{00}$:
\begin{equation}\label{pre-LK}
\big(\LKB v\big)(n)=\lambdaK_n\big(v(n+1)-v(n)\big)+\muK_n \big(v(n-1)\un_{\{n>1\}}-v(n)\big)
\end{equation}
for $n\in\integers$.
The idea is to `localize' this operator either around $n=\lfloor K\xf\rfloor$ or $n=1$, which corresponds in the dynamical system to the fixed point $\xf$
or the fixed point $0$. 
A natural idea would be to `cut' the operators in order to differentiate these two dynamics. 
However the main difficulty is that the two different pieces involve different scales and different function spaces.
Since we don't know how to cope with this problem at the level of operators, we work at the level of the eigenfunctions
that we will split  on different regions, and study their asymptotic dependence on $K$ in each region. 

To have an idea of the different scales involved in the problem, let us first study the asymptotic behaviour of the birth and death rates. 
By Taylor expansion around $K\xf$ we have
\[
\lambdaK_{n}=K\,b\big(\xf\big)+(n-\Kxf)\,b'\big(\xf\big)+\mathcal{O}\left(\frac{(n-\Kxf)^{2}}{K}\right)
\]
and 
\[
\muK_{n}=K\,d\big(\xf\big)+(n-\Kxf)\,d'\big(\xf\big)+\mathcal{O}\left(\frac{(n-\Kxf)^{2}}{K}\right).
\]
Taking $v^{\psK}(n)=u\big((n-\Kxf)/\sqrt{K}\big)$ with $u\in\mathscr{D}$, we get
\[
\big(\LKB v^{\psK}\big)(n)=\OU u\left(\frac{n-\Kxf}{\sqrt{K}}\right)+\mathcal{O}\left(\frac{1}{\sqrt{K}}\right)
\]
where
\begin{equation}\label{def-OU-generator}
\OU f(x)= b(\xf) f''(x)+ \,\big(b'(\xf)-d'(\xf)\big)\,x f'(x)
\end{equation}
is the generator of an Ornstein-Uhlenbeck process on $\real$ which satisfies the stochastic differential equation
$\dd X_t=\big(b'(\xf)-d'(\xf)\big)X_t \dd t + \sqrt{2b(\xf)}\, \dd B_t$, where $(B_t)_{t\geq 0}$ is a one-dimensional Brownian motion.
It is well known (see Remark \ref{rem:OUHO}) that the spectrum of $\OU$ is $-S_1$ where
\begin{equation}\label{def-S1} 
S_1=\big(d'(\xf)-b'(\xf)\big)\entiers
\end{equation}
in the space $L^2\Big(\sqrt{\frac{d'(\xf)-b'(\xf)}{2\pi b(\xf)}}\,\e^{-\frac{(d'(\xf)-b'(\xf)}{2b(\xf)}x^2}\dd x\Big)$.
Now we look at $n$ near $1$. By Taylor expansion, we have
\[
\lambdaK_{n}=n\left(b'(0)+\mathcal{O}\left(\frac{n}{K}\right)\right)\quad\text{and}\quad \muK_{n}=n\left(b'(0)+\mathcal{O}\left(\frac{n}{K}\right)\right).
\]
If $v\in c_{00}$, then we get
\[
\left\| \LKB v -\OUP v\right\|_{\ldeux}\leq \frac{\Oun}{K}
\] 
where
\[
\big(\OUP v\big)(n)=\tlamz\,n\,\big(v(n+1)-v(n)\big)+\tmuz\, n\,\big(v(n-1)\,\un_{\{n>1\}}-v(n)\big)
\]
which is the generator of a (continuous-time) binary branching process  killed at $0$.
We shall prove later on that, in a weighted $\ell^2$ space defined below, the spectrum of $\OUP$ is $-S_2$ where
\begin{equation}\label{def-S2}
S_2=\big(b'(0)-d'(0)\big)\integers.
\end{equation}

The previous observations suggest that the limit of the spectrum of the generator of the birth-and-death process $(X^{\psK}_t)_{t\geq 0}$, in an appropriate space, is
\[
\big(d'(0)-b'(0)\big)\integers \bigcup \big(b'(\xf)-d'(\xf)\big)\entiers.
\]
Notice that all the elements of this set are negative and this is not a disjoint union in general. The logistic model is an example illustrating this since
$d'(0)-b'(0)=b'(\xf)-d'(\xf)=\mu-\lambda$, so we will have asymptotic double eigenvalues in this case. 

We will prove that the limit of the spectrum of $\LKB$ is obtained from the explicit spectra of the above two operators.
Notice that one is differential operator and the other one is a finite-difference operator.
This is a reverse situation with respect to numerical analysis where the spectrum of limiting differential operators are obtained
from the knowledge of the spectrum of finite-difference operators. See for instance \cite{Chatelin}.

\subsection{Main result}

For each $K\in \integers$, the sequence of numbers
\[
\piK_n:=\frac{\lambdaK_1\cdots \lambdaK_{n-1}}{\muK_1\cdots \muK_n}, \, n\geq 2\quad\text{and}\quad \piK_1:=\frac{1}{\muK_1}
\]
naturally shows up in the study of birth-and-death processes.
We will give below a set of assumptions on the functions $b$ and $d$, defining the differential equation \eqref{the-edo}, ensuring that the process
reaches $0$ in finite time with probability one, that the mean-time to extinction is finite, and that the quasi-stationary distribution exists and is unique.

Let $\ell^2(\piK)$ be the space of $\complex$-valued sequences $(v_n)_{n\geq 1}$ such that 
\[
\sum_{n\geq 1} |v_n|^2\piK_n<\infty\,.
\]
This is a Hilbert space when endowed with the scalar product $\langle v,w\rangle_{\piK}:=\sum_{n\geq 1} \bar{v}_nw_n \piK_n$.

We know from \cite{CCM1} that the operator $\LKB$ is closable in $\ell^2(\piK)$. 
The closure (which we denote by the same symbol) is self-adjoint and has a compact resolvent, hence its spectrum is discrete, composed of simple eigenvalues which are 
negative real numbers, and the corresponding eigenvectors are orthogonal. We normalize these eigenvectors and we can assume that they are real (since they are
defined by a second-order real recurrence relation whose solution is determined by choosing the first element).
We write
\begin{equation}\label{def-eigen-LK}
L_{\sK}\psi^{\psK}_j=-\rho^{\psK}_j \psi^{\psK}_j
\end{equation}
where we order the eigenvalues $-\rho^{\psK}_j$ in decreasing order as $j$ increases. 
To emphasize that all operators considered in this paper are negative, we have decided to write their eigenvalues under the form $-\rho$, with $\rho>0$.

As shown in \cite{CCM1}, the quasi-stationary distribution exists, is unique, and given by
\begin{equation}
\label{QSD}
\nuK(\{n\})=\frac{\piK_n\psi^{\psK}_0(n)}{\langle \psi^{\psK}_0,\mathbf{1}\rangle_{\piK}},\;n\in\integers
\end{equation}
where $\mathbf{1}=(1,1,\ldots)$.
Note that it also follows from Theorem 3.2 and Lemma 9.3 in \cite{CCM1} that there exists $D>1$ such that for all $K\in\integers$, we have
$D^{-1}\leq \psi_0^{\psK}\leq D$. Therefore, the Hilbert spaces $\ldeux(\nuK)$ and $\ldeux(\piK)$ are isomorphic.

Let  $S^{\psK}_t f(n):=\esperance_n\Big( f(X^{\psK}_t) \un_{\{\TK>t\}}\Big)$ ($t\geq 0$) be the semigroup of the killed process, where $n\in\integers$ and $f\in \ell^\infty$. The following result justifies that we look for the spectrum of $L_{\sK}$ in  $\ldeux(\piK)$ . 

\begin{proposition}
The semigroup $(S^{\psK}_t)_{t\geq 0}$, defined on $\ell^\infty$, extends to a $C_0$-contraction semigroup on $\ldeux(\nuK)$.
\end{proposition}
We refer to \cite{yosida} for definitions and properties of $C_0$-contraction semigroups.

\begin{proof}
We follow the argument of Proposition 8.1.8 p. 162 in \cite{BL}.  Since $\nuK$ is a quasi-stationary distribution, for $f\in c_{00}$ and $t\geq 0$, we have
\begin{align*}
\int |S^{\psK}_t f|^2 \dd\nuK 
& \leq \int \dd\nuK(n)\, \esperance_n \Big( \big|f(X^{\psK}_t)\big|^2 \un_{\big\{\TK>t\big\}}\Big)\\
& = \e^{-\rho^{\psK}_0 t} \int |f|^2 \dd\nuK.
\end{align*}
Therefore
\[
\| S^{\psK}_t f\|_{\ldeux(\nuK)}\leq \e^{-\frac{\rho^{\psK}_0 \mathlarger{t}}{2}} \|f\|_{\ldeux(\nuK)},\;t\geq 0.
\]
Since $c_{00}$ is dense in $\ldeux(\nuK)$, we get
\[
\| S^{\psK}_t \|_{\ldeux(\nuK)}\leq \e^{-\frac{\rho^{\psK}_0 \mathlarger{t}}{2}}, \; t\geq 0.
\]
This implies that $(S^{\psK}_t)_{t\geq 0}$ extends to a contraction semigroup in $\ldeux(\nuK)$. 
Now, since $\proba_n(X^{\psK}_t=m,\TK>t)\to \delta_{n,m}$, as $t\to 0$, then for any $f\in c_{00}$, $S^{\psK}_t f\to f$ pointwise,
hence by dominated convergence we obtain $S^{\psK}_t f \to f$ in $\ldeux(\nuK)$ as $t\to 0$. 
The proposition follows from the contraction property obtained above and the fact that $c_{00}$ is dense in $\ldeux(\nuK)$.
\end{proof}
Observe that the same result holds in fact for any $1\leq p <\infty$ (with a similar proof).

Recall from  \cite{CCM1} that  we also have
\[
\proba_{\nuK}(\TK>t)=\e^{-\rho^{\psK}_0 t},\;t>0
\]
and the mean-time to extinction starting from $\nuK$ is
\begin{align*}
\MoveEqLeft \expectation_{\nuK}\big[\TK\big]\\
&=\frac{1}{\rho^{\psK}_0}= \frac{\sqrt{2\pi}\,\exp\bigg(K\mathlarger{\int}_{0}^{\xf} \log\frac{b(x)}{d(x)}\,\dd x\bigg)}{b(\xf)\left(\sqrt{\frac{b(1/\sK)}{d(1/\sK)}}-\sqrt{\frac{d(1/\sK)}{b(1/\sK)}}\right)\! 
\sqrt{K H''(\xf)}}
\!\left(\!1\!+\!\mathcal{O}\!\left(\frac{(\log K)^3}{\sqrt{K}}\right)\!\!\right).
\end{align*}
(See the next section for the definition of $H$.)
In the logistic model this gives (recall that $x_*=\lambda-\mu$)
\[
\expectation_{\nuK}\big[\TK\big]=
\frac{\sqrt{2\pi}\,\mu\,\exp\Big(K\big(\lambda-\mu+\mu\log\frac{\mu}{\lambda}\big)\Big)}{(\lambda-\mu)^2\sqrt{K}}
\!\left(\!1\!+\!\mathcal{O}\!\left(\frac{(\log K)^3}{\sqrt{K}}\right)\!\!\right).
\]
Thus $\rho^{\psK}_0$ is exponentially small in $K$, and we also proved in \cite{CCM1} that the `spectral gap' satisfies
(see Theorem 3.3 in \cite{CCM1})
\begin{equation}
\label{lower-bound-spectral-gap}
\rho^{\psK}_1-\rho^{\psK}_0 \geq \frac{\Oun}{\log K}, \; K\in\mathds{Z}_{{\scriptscriptstyle >1}}.
\end{equation}
This lower bound goes to $0$ as $K\to+\infty$. A noteworthy consequence of the main results of this paper (see Corollary \ref{trou}) is 
that the spectral gap does not close when $K$ tends to infinity, contrary to what could have been  suspected from the lower bound in \eqref{lower-bound-spectral-gap}.

In fact we will fully describe the asymptotics of all eigenvalues, which is the content of our main theorem. To state it, 
we need to order $S_1\cup S_2$ to take care of possible multiplicities. Recall that $S_{1}$ and $S_{2}$ have been defined in \eqref{def-S1} and \eqref{def-S2} 
and that they are positive sequences. 
This is done through the definition of a non-decreasing infinite (positive) sequence $(\eta_n)_{n\geq 0}$. 
Let $\eta_0=0$. We construct this sequence recursively as follows.
\begin{itemize}
\item[$\diamond$]
If $\eta_n\in S_1\Delta S_2$, then $\eta_{n+1}=\min\{\eta: \eta \in S_1\cup
S_2: \eta>\eta_n\}$. 
\item[$\diamond$]
If $\eta_n\in S_1\cap S_2$, then
\begin{itemize}
\item[$\bullet$]
If $\eta_{n-1}=\eta_n$, then $\eta_{n+1}=\min\{\eta: \eta \in S_1\cup S_2: \eta>\eta_n\}$.
\item[$\bullet$]
If $\eta_{n-1}<\eta_n$, then $\eta_{n+1}=\eta_n$.
\end{itemize}
\end{itemize}

The main result of this paper, whose assumptions will be stated in Section 2, is the following. 
\begin{theorem}[Convergence of the spectrum]\label{main-theorem}
\leavevmode\\
The spectrum of $L_{\sK}$ in $\ell^2(\piK)$ converges pointwise to $(-\eta_n)_{n\geq 0}$ when $K$ tends to infinity. In other words
\[
\lim_{K\to+\infty}\rho^{\psK}_j = \eta_j, \;\forall j\in\entiers.
\]
\end{theorem}

\begin{corollary}\label{trou}
The spectral gap $\rho^{\psK}_1-\rho^{\psK}_0$ converges to
\[
\min\big\{b'(0)-d'(0),\,d'(\xf)-b'(\xf)\big\}\;.
\]
\end{corollary}

Let us give some examples.
In the logistic model, we have $d'(0)-b'(0)=b'(\xf)-d'(\xf)=\mu-\lambda$ (asymptotic double eigenvalues), and the spectral gap is equal to $\lambda-\mu$.
Another example is the Ayala-Gilpin-Ehrenfeld model \cite{AGE} defined by $b(x)=\lambda x$, $d(x)=x(\mu+x^\theta)$ where $\theta\in (0,1)$ is a parameter,
and $\lambda>\mu$.
In this case, $d'(0)-b'(0)=\mu-\lambda$, $\xf=(\lambda-\mu)^{1/\theta}$, and $b'(\xf)-d'(\xf)=\theta(\mu-\lambda)$, so the spectral gap is $\theta(\lambda-\mu)$.
Yet another example is Smith's model \cite{smith} defined by $b(x)=\lambda x/(1+x)$, $d(x)=(x(\mu+x))/(1+x)$, where $\lambda>\mu$. One easily finds
$d'(0)-b'(0)=\mu-\lambda$, $\xf=\lambda-\mu$, and $b'(\xf)-d'(\xf)=(\mu-\lambda)/(1+\lambda-\mu)$, so the spectral gap is $(\lambda-\mu)/(1+\lambda-\mu)$.

\subsection{Consequences on relaxation times}

Recall  that the spectral gap is the inverse of the relaxation time to the quasi-stationary distribution, namely
for $\psi\in\ell^2(\piK)$,  $t>0$ and $K\in\integers$ we have 
\[
\left\|\e^{\rho^{\psK}_0 t} \, S^{\psK}_t \psi - \psi_0^{\psK}\langle \psi_0^{\psK},\mathbf{1} \rangle_{\piK} \nuK(\psi)\right\|_{\ell^2(\piK)}
\leq  \|\psi\|_{\ell^2(\piK)}\,\e^{-\big(\rho^{\psK}_1-\rho^{\psK}_0\big)t}.
\]

From Corollary \ref{trou}, it turns out that the relaxation time converges to a finite limit as $K$ tends to infinity. 

We can also characterize the decay of correlations for the  so-called   $Q$-process, namely the birth-and-death process conditioned on survival. Recall that 
the $Q$-process is the irreducible Markov process with state space $\integers$, defined by the semigroup 
\[
R^{\psK}_t g=\e^{\rho_0^{\psK}t} \frac{1}{\psi_0^{\psK}} \, S^{\psK}_t\big(g\psi_0^{\psK}\big).
\]
It satisfies $R^{\psK}_t \un_{\{n\geq 1\}}= \un_{\{n\geq 1\}}$, and its unique invariant distribution $\mathfrak{m}^{\psK}$, defined by
${\big(R_t^{\psK}\big)}^\dagger\mathfrak{m}^{\psK}=\mathfrak{m}^{\psK}$, 
is related to $\nuK$ by $\mathfrak{m}^{\psK}(g)=\nuK(\psi_0^{\psK} g)$ where $\nuK$ has been defined in \eqref{QSD}. 
Indeed we have  $g \in \ldeux(\mathfrak{m}^{\psK})$ if and only if $\psi_0^{\psK}g \in \ldeux(\piK)$, and 
\[
\| g\|^2_{\ldeux(\mathfrak{m}^{\psK})}=\frac{\| \psi_0^{\psK}g\|^2_{\ldeux(\piK)}}{\langle \psi_0^{\psK},\mathbf{1}\rangle_{\piK}}.
\]
Hence, we get the following result.
\begin{proposition}
Let $g\in \ldeux(\mathfrak{m}^{\psK})$. Then for all $t> 0$
\[
\big\| R^{\psK}_t g -\langle \psi_0^{\psK},\mathbf{1}\rangle_{\piK} \mathfrak{m}^{\psK}(g)\big\|_{\ldeux(\mathfrak{m}^{\psK})}
\leq \|g\|_{\ldeux(\mathfrak{m}^{\psK})}\, \e^{-(\rho^{\psK}_1-\rho^{\psK}_0)\,t}.
\]
Furthermore, for $g_1,g_2\in\ldeux(\mathfrak{m}^{\psK})$ and for all $t>0$, we have
\begin{align*}
\MoveEqLeft[6] \bigg|\int R^{\psK}_t g_1\cdot g_2 \, \dd\mathfrak{m}^{\psK}- \int g_1\, \dd\mathfrak{m}^{\psK}\int g_2\, \dd\mathfrak{m}^{\psK}\bigg|\\
& \leq \|g_1\|_{\ldeux(\mathfrak{m}^{\psK})}\|g_2\|_{\ldeux(\mathfrak{m}^{\psK})} \,\e^{-\big(\rho^{\psK}_1-\rho^{\psK}_0\big)t}.
\end{align*}
\end{proposition}
As before, the rate of decay of correlations converges when $K$ goes to infinity. 

\subsection{Organization of the paper}

The proof of the Theorem \ref{main-theorem} relies on two results stated in Section \ref{preuvemain}. Theorem \ref{liminf} ensures that the 
set $S_{1}\cup S_{2}$ is contained in the set of accumulation points of the eigenvalues of $\L_{\sK}$ when $K$ tends to infinity. 
The proof is based on the construction of quasi-eigenvectors and is given in Section \ref{proof-maintheorem}.
The second result is Theorem \ref{prsquelafin} which ensures that all the previous accumulation points are contained in $S_{1}\cup S_{2}$ taking care of 
eventual multiplicities. Its proof, given in section \ref{proof-maintheorem}, relies on two propositions. The first one (Proposition \ref{strucvp}) is the splitting of the 
eigenvectors of  $\LKl$  into two dominant parts, one localised near the origin, the other one near $\lfloor K\xf\rfloor$. 
The second (Proposition \ref {decoupe}) relies on compactness arguments of each piece of the previous splitting. Section \ref{auxiliaires} collects various 
auxiliary results (some of more general nature).

One of the main difficulties of the proof is that the two pieces of the spectrum correspond to limiting operators which are obtained at
different scales and leave in different function spaces. 

\section{Standing assumptions}\label{sec:assumptions}

We work under the assumptions of \cite{CCM1} which we recall for convenience. 

The functions $b,d:\mathds{R}_+\to\mathds{R}_+$ defining the differential equation \eqref{the-edo} are supposed to be such that  
\[
b(0) = d(0)=0
\]
and  the functions $x\mapsto b(x)/x$ and $x\mapsto d(x)/x$ are defined on $\mathds{R}_+$ and assumed to be positive, twice differentiable and increasing (in particular
the sequences $(\lambdaK_{n})_n$ and $(\muK_{n})_n$ defined in \eqref{def-lambdaKnmuKn} are increasing for each $K$).

We start by the biologically relevant assumptions:
\begin{itemize}
\renewcommand{\labelitemi}{$\centerdot$}
\item
$\lim_{x\to+\infty} \frac{b(x)}{d(x)}=0$ (deaths prevail over births for very large densities).
\item
$b'(0)>d'(0)>0$ (at low density births prevail).
\item
There is a unique $\xf>0$ such that $b(\xf)=d(\xf)$, so $\xf$ is the only positive fixed point of \eqref{the-edo}.
\end{itemize}
We assume that $ b'(\xf)\neq d'(\xf)$ (genericity condition). The remaining (technical) assumptions are the following:
\begin{itemize}
\renewcommand{\labelitemi}{$\centerdot$}
\item
$\mathlarger{\int}_{\frac{\xf}{2}}^{+\infty} \frac{\dd x}{d(x)}<+\infty$ and $\sup_{x\in\mathds{R}_+} \big(\frac{d'(x)}{d(x)}-\frac{1}{x}\big)<+\infty$. 
\item
The function $x\mapsto \log\frac{d(x)}{b(x)}$ is increasing on $\mathds{R}_+$.
\item
The function $H:\mathds{R}_+\to\mathds{R}$ defined by $H(x)=\int_{\xf}^x  \log\frac{d(s)}{b(s)}\dd s$ is three times differentiable, and $\sup_{x\in\mathds{R}_+}(1+x^2)|H'''(x)|<+\infty$.
\end{itemize}
The assumptions imply that $0$ is a repulsive (or unstable) fixed point of \eqref{the-edo}, whereas $\xf$ is an attractive (or stable) one, that is, $b'(\xf)<d'(\xf)$. It also follows that $H''(\xf)>0$.
These assumptions are satisfied for many classical examples.  

As explained in \cite{CCM1}, the above conditions imply the following properties:
\begin{itemize}
\renewcommand{\labelitemi}{$\centerdot$}
\item $\sum_{n\geq 1} (\lambdaK_n \piK_n)^{-1}=+\infty$, which implies that the process reaches $0$ in finite time with probability one.
\item
$\sum_{n\geq 1} \piK_n<+\infty$, which implies finiteness of the mean time to extinction.
\item $\sum_{n\geq 1} (\lambdaK_n \piK_n)^{-1} \sum_{i\geq n+1} \piK_i<+\infty$, which is a necessary and sufficient condition for existence and uniqueness of the quasi-stationary distribution.
\end{itemize}
We add a last condition to the previous ones, namely
\begin{equation}\label{hipopo}
\lim_{x\to\infty}\frac{\log b(x)}{x}=\lim_{x\to\infty}\frac{\log d(x)}{x}=0.
\end{equation}
We could avoid it makes our life easier and we don't have any natural example which does not satisfy it.

\section{Proof of Theorem \ref{main-theorem}}\label{preuvemain}

\subsection{Some useful operators}

Instead of working with $L_{\sK}$ on the weighted Hilbert space $\ell^2(\piK)$, we find more convenient to work on the `flat' Hilbert space $\ldeux$.
We introduce the conjugated operator
\begin{equation*}
\LKl=(\Pi^{\psK})^{\frac{1}{2}}\LKB\,(\Pi^{\psK})^{-\frac{1}{2}}
\end{equation*}
where $\Pi^{\psK}$ denotes the mutiplication operator 
\[
\Pi^{\psK} v(n)=\piK_{n}v(n)
\]
for $v\in c_{00}$ and $n\in\integers$. One can check that
\begin{align*}
& \big(\LKl v\big)(n)\\
& =\sqrt{\lambdaK_{n}\,\muK_{n+1}}\,v(n+1)+\sqrt{\lambdaK_{n-1}\,\muK_{n}}\,v(n-1)\,\un_{\{n>1\}}-\big(\lambdaK_{n}+\muK_{n}\big)\,v(n)
\end{align*}
for $n\in\integers$.
We denote also by $\LKl$ its closure in $\ldeux$ and by $\Dom(\LKl)$ its domain, and we have
\begin{equation*}
\LKl \phi^{\psK}_j=-\rho^{\psK}_j \phi^{\psK}_j
\end{equation*}
where the eigenvalues $-\rho^{\psK}_j$ are the same as for $L_{\sK}$ (cf. \eqref{def-eigen-LK}), and $\phi^{\psK}=\big(\Pi^{\psK}\big)^{\frac{1}{2}}\psi^{\psK}$.
To capture the behavior of the eigenvectors of $\LKl$ near $\lfloor K\xf\rfloor$ at scale $\sqrt{K}$, we are going to embed $\ldeux$ into $\Ldeux$.
For this purpose we define for each $K\in\integers$ the functions
\[
\eK{n}(x)=K^{\frac{1}{4}}\,\un_{I^{{\scriptscriptstyle (K)}}_{n} }(x),\; x\in\real,\;n\in\integers
\]
where 
\[
I^{\scriptscriptstyle (K)}_{n}=\left[\frac{n-0.5}{\sqrt{K}}-\xf\sqrt{K}, \frac{n+0.5}{\sqrt{K}}-\xf\sqrt{K}\right[.
\]
The functions $\eK{n}$ are orthogonal and of norm one in $\Ldeux$. They form a basis of a sub-Hilbert space $\HK_{\sK}$ of piecewise constant 
functions in $\Ldeux$.
We define two maps denoted by  $Q_{\sK}$ and $P_{\sK}$ as follows:
\begin{equation*}
Q_{\sK}: \ldeux\to\Ldeux,\quad
Q_{\sK} u(x)=\sum_{n\geq 1}u(n)\,\eK{n}(x)\
\end{equation*}
and 
\begin{equation}\label{PP}
P_{\sK}:\Ldeux \to \ldeux,\quad
P_{\sK}f(n)=\int f(x)\,\eK{n}(x)\,\dd x\;, n\in\integers.
\end{equation}
We will use the following properties of $P_{\sK}$ and $Q_{\sK}$ stated as two lemmas.  
\begin{lemma}\label{isometry}
For each $K\in\integers$, the map $Q_{\sK}$ is an isometry between $\ldeux$ and $\HK_{\sK}$. 
\end{lemma}
The proof of this lemma is left to the reader.
\begin{lemma}\label{isomQP}
Let $f\in C^{1}(\real)$ and assume  that there exists $a>0$ and $A>0$ such that 
\[
\big|f(x)\big|+\big|f'(x)\big|\le A\,\e^{-a\,|x|},\, x\in\real.
\]
Then 
\begin{itemize}
\item[\textup{(i)}] $\lim_{K\to\infty}\big\|f-Q_{\sK}P_{\sK}f\big\|_{\Ldeux}=0$\,.
\item[\textup{(ii)}] $\lim_{K\to\infty}\big\|P_{\sK}f\big\|_{\ldeux}=\big\|f\big\|_{\Ldeux}$\,.
\end{itemize}
\end{lemma}
\begin{proof}
Let us prove (i). 
We have $Q_{\sK}P_{\sK}f(x)=e^{\psK}_{n(x)}(x)\;\int e^{\psK}_{n(x)}(y)\,f(y) \,\dd y$
for $(n(x)-0.5)/\sqrt{K}-\xf\le x \le (n(x)+0.5)/\sqrt{K}-\xf$.
We get from our hypothesis
\[
Q_{\sK}P_{\sK}f(x)=f(x)+\mathcal{O}\big(K^{-\frac{1}{2}}\big)\,\e^{-a|x|}
\]
and the result follows.\newline
We now prove (ii).
From the isometric property of $Q_{K}$ on the space of piecewise functions $\HK_{\sK}$, we get
\[
\big\|P_{\sK}f\big\|_{\ldeux}=\big\|Q_{\sK}P_{\sK}f\big\|_{\Ldeux}
\]
and the result follows from (i).
\end{proof}

We now introduce the operator 
\begin{equation*}
\LKL= Q_{\sK}\LKl P_{\sK}
\end{equation*}
and we keep the same notation for its closure in $\Ldeux$. Since  $\LKL$ when acting on $\HK_{\sK}$ is conjugated to $\LKl$, we have 
\[
\LKL\varphi^{\psK}_j=-\rho^{\psK}_j \varphi^{\psK}_j.
\]
We will prove in the next proposition that the operator $\LKL$ converges weakly, when $K\to+\infty$, to the operator
\begin{align}
\label{harm}
\MoveEqLeft[4] \HO f(x)=\\
& b(\xf)\,\frac{\dd^{2}f(x)}{\dd x^{2}}-\frac{\big(d'(\xf)-b'(\xf)\big)^{2}}{4b(\xf)}\,x^{2}f(x)+\frac{d'(\xf)-b'(\xf)}{2}f(x) .
\nonumber
\end{align}

\begin{proposition}\label{convS}
Let $f\in C^{3}(\real)$ and assume  that there exist $a>0$ and $A>0$ such that
\[
\sum_{j=0}^{3}\big|f^{(j)}(x)\big|\le A\, \e^{-a\,|x|},\, x\in\real.
\]
Then 
\[
\lim_{K\to\infty}\big\|\LKL f-\HO f\big\|_{\Ldeux}=0\,.
\]
\end{proposition}
\begin{proof}
By the assumption made on $f$, it follows easily that
\[
\lim_{K\to\infty}\big\|\un_{\{|\,\cdot\, |>(\log K)^{2}\}}\;\HO f\big\|_{\Ldeux}=0.
\]
We have
\begin{align}\label{LKL}
& \LKL f(x)=
\sum_{n\geq 1} e^{\psK}_{n}(x)\,\left(\sqrt{\lambdaK_{n}\,\muK_{n+1}}\,\int e^{\psK}_{n+1}(y)\,f(y)\,\dd y\right. \\
\nonumber
& \left.+\sqrt{\lambdaK_{n-1}\,\muK_{n}}\,\un_{\{n>1\}} \,\int e^{\psK}_{n-1}(y)\,f(y)\,\dd y -\big(\lambdaK_{n}+\muK_{n}\big)\,\int e^{\psK}_{n}(y)\,f(y)\,\dd y\right).
\end{align}
It follows easily from the assumption made on $f$ and assumption \eqref{hipopo} that 
\[
\lim_{K\to\infty}\big\|\un_{\{|\,\cdot\, |>(\log K)^{2}\}}\LKL f\big\|_{\Ldeux}=0\,. 
\]
Therefore we only have to consider $|x|\le (\log K)^{2}$.
Note also that for such an   $x$, the sum in \eqref{LKL} reduces to one element for $K$ large enough, namely $n=n(x) = \big\lfloor K\xf +\sqrt{K} x +\frac{1}{2}\big\rfloor$. 
For $x\in\real$, we have
\begin{align*}
\LKL f(x) 
& = \e^{\psK}_{n(x)}(x)\left(\sqrt{\lambdaK_{n(x)}\muK_{n(x)+1}}\,\int e^{\psK}_{n(x)+1}(y)f(y)\,\dd y\right. \\
& \left.\hspace{2.1cm}+\sqrt{\lambdaK_{n(x)-1}\,\muK_{n(x)}} \int e^{\psK}_{n(x)-1}(y)f(y)\,\dd y\right.\\
& \left. \hspace{2.1cm}-\,\big(\lambdaK_{n(x)}+\muK_{n(x)}\big)\int e^{\psK}_{n(x)}(y) f(y)\,\dd y\right)\\
&= e^{\psK}_{n(x)}(x)\left(\sqrt{\lambdaK_{n(x)}\,\muK_{n(x)+1}}\int e^{\psK}_{n(x)}(y)f\left(y-\frac{1}{\sqrt{K}}\right)\dd y\right.\\
& \left.\hspace{2cm}+\sqrt{\lambdaK_{n(x)-1}\,\muK_{n(x)}} \int e^{\psK}_{n(x)}(y) f\left(y+\frac{1}{\sqrt{K}}\right)\dd y\right.\\
& \left. \hspace{2cm}-\big(\lambdaK_{n(x)}+\muK_{n(x)}\big)\int e^{\psK}_{n(x)}(y)f(y)\,\dd y\right).
\end{align*}
Now we have
\begin{align*}
\MoveEqLeft \int e^{\psK}_{n(x)}(y)\,f(y)\,\dd y \\
& =K^{\frac{1}{4}}\int_{-\frac{1}{2\sqrt{K}}}^{\frac{1}{2\sqrt{K}}}f\left(\frac{n(x)-\Kxf}{\sqrt{K}}-s\right)\,\dd s\\
&=K^{-\frac{1}{4}}\,f\left(\frac{n(x)-\Kxf}{\sqrt{K}}\right)+\frac{K^{-\frac{5}{4}}}{24}f''\left(\frac{n(x)-\Kxf}{\sqrt{K}}\right)+\mathcal{O}\Big(K^{-\frac{7}{4}}\Big)
\end{align*}
where the error term is unifom in $x$. Similarly
\begin{align*}
\MoveEqLeft \int e^{\psK}_{n(x)}(y)\,f\left(y\pm \frac{1}{\sqrt{K}}\right)\dd y \\
&= \int e^{\psK}_{n(x)}(y)\,f(y)\,\dd y \pm\frac{1}{\sqrt{K}}\int e^{\psK}_{n(x)}(y)\,f'(y)\,\dd y \\
&\quad +\frac{1}{2K}\int e^{\psK}_{n(x)}(y)\,f''(y)\,\dd y+ \mathcal{O}\Big(K^{-\frac{7}{4}}\Big)\\
& =K^{-\frac{1}{4}}\,f\left(\frac{n(x)-\Kxf}{\sqrt{K}}\right) \pm K^{-\frac{3}{4}} \,f'\left(\frac{n(x)-\Kxf}{\sqrt{K}}\right)\\
& \quad +\frac{13\,K^{-\frac{5}{4}}}{24} f''\left(\frac{n(x)-\Kxf}{\sqrt{K}}\right)+\mathcal{O}\big(K^{-\frac{7}{4}}\big).
\end{align*}
Recall that 
\[
\lambdaK_{n}=K b\left(\frac{n}{K}\right)\quad\text{and}\quad \muK_{n}=K d\left(\frac{n}{K}\right)
\]
hence
\[
\lambdaK_{n}=K b\big(\xf\big)+(n-\Kxf) \,b'\big(\xf\big)+\mathcal{O}\left(\frac{(n-\Kxf)^{2}}{K}\right)
\]
and 
\[
\muK_{n}=K d\big(\xf\big)+(n-\Kxf)\,d'\big(\xf\big)+\mathcal{O}\left(\frac{(n-\Kxf)^{2}}{K}\right).
\]
After a tedious but straightforward computation, we obtain that
\begin{align*}
\LKL f(x)
&=\HO f\left(\frac{n(x)-\Kxf}{\sqrt{K}}\right)+\mathcal{O}\big(K^{-\frac{1}{4}}\big)\\
&=\HO f(x)+\mathcal{O}\big(K^{-\frac{1}{4}}\;(\log K)^{4}\big)
\end{align*}
and the error term is uniform in $|x|\le (\log K)^{2}$.
We get
\[
\big\|\un_{\{|\,\cdot\, |\leq (\log K)^{2}\}}\;\big(\LKL f-\HO f\big)\big\|_{\Ldeux}=\mathcal{O}\big(K^{-\frac{1}{4}}\;(\log K)^{6}\big)
\]
and the result follows.
\end{proof}

\begin{remark}\label{rem:OUHO}
Let us recall the relationship between the generator of the Ornstein-Uhlenbeck process \eqref{def-OU-generator} and \eqref{harm} which is, up to a minus sign and a shift, the Schr\"odinger operator
for the quantum harmonic oscillator. We refer to {\em e.g.} \cite[Chapter 3]{babusci-et-al} or \cite[Sections 4.4 and 4.9]{pavliotis}.
In $\Ldeux$, the eigenvalues of\, $\HO$ are $-\big(d'(\xf)-b'(\xf))n$, $n\in\entiers$, and the corresponding eigenfunctions are
\begin{align}\label{eigenstates-HO}
&\psi_n(x)=\\
&\frac{1}{\sqrt{2^n n!}} \left(\frac{\big(d'(\xf)-b'(\xf)\big)^2}{2\pi b(\xf)}\right)^{\frac{1}{4}} \e^{-\frac{d'(\xf)-b'(\xf)}{4b(\xf)}x^2}H_n\left(\sqrt{\frac{d'(\xf)-b'(\xf)}{2b(\xf)}}x\right)
\nonumber
\end{align}
where $(H_n)_n$ is the family of the physicists' Hermite polynomials defined by
\[
H_n(x)=(-1)^n \e^{x^2} \frac{\dd^n}{\dd x^n}\e^{-x^2}.
\]
One can check that $\HO$ is conjugated to the generator of the Ornstein-Uhlenbeck process \eqref{def-OU-generator} 
acting on $L^2\Big(\sqrt{\frac{d'(\xf)-b'(\xf)}{2\pi b(\xf)}}\,\e^{-\frac{(d'(\xf)-b'(\xf)}{2b(\xf)}x^2}\dd x\Big)$ in the following way:
$\frac{1}{\psi_0}\HO(\psi_0 f)=\OU f$.
\end{remark}

In Proposition \ref{convM0} we prove that the operator $\LKl$ converges weakly, when $K$ tends to infinity, to the operator $\PO$ defined for $v\in c_{00}$ by 
\begin{align}
\label{le-M0}
& \big(\PO v\big)(n)=\\
\nonumber
& \sqrt{\tlamz\,\tmuz\,n\,(n+1)} \,v(n+1)+\sqrt{\tlamz\,\tmuz\,n\,(n-1)}\,v(n-1)\,\un_{\{n>1\}}\\
\nonumber
& -n\,(\tlamz+\tmuz) v(n)).
\end{align}
Here again we denote the operator on $c_{00}$ and its closure by the same letter.

\begin{proposition}\label{convM0}
Let $u\in c_{00}$. Then
\[
\lim_{K\to\infty}\LKl u=\PO u
\]
where $\PO$ is defined in \eqref{le-M0}.
\end{proposition}
\begin{proof}
Follows from the fact that for each fixed $n$
\[
\lim_{K\to\infty}\lambdaK_{n}=\tlamz\,n
\qquad
\mathrm{and} \qquad
\lim_{K\to\infty}\muK_{n}=\tmuz\,n\;.
\]
\end{proof}


\subsection{Steps of the proof of Theorem \ref{main-theorem}}

The proof of Theorem \ref{main-theorem} relies on the following two theorems whose proofs are postponed to Section \ref{proof-maintheorem}.
Recall that for any fixed $K$, the spectrum $\spectre(\mathcal{L}_{\sK})$ is discrete, and let
\[
G=\bigcup_{j=0}^\infty\big(\rhoK{j}\big)^{\mathrm{{\scriptscriptstyle acc}}}
\]
where $\big(\rhoK{j}\big)^{\mathrm{{\scriptscriptstyle acc}}}$ is the set of accumulation points of $\big(\rhoK{j}\big)$ when $K\to+\infty$.  

\begin{theorem}\label{liminf}
We have
\[
S_1\cup S_2 \subset G
\]
where $S_1$ and $S_2$ are defined in \eqref{def-S1} and \eqref{def-S2}.
\end{theorem}
This theorem is proved in Section \ref{proof-liminf}.

\begin{corollary}\label{cfini}
For every fixed $j$ we have
\[
\limsup_{\sK\to+\infty} \rho_j^{\psK}<+\infty.
\]
\end{corollary}
\begin{proof}
We proceed by contradiction.
Assume that there exists $j_0$ such that
\[
\limsup_{\sK\to+\infty} \rho_{j_0}^{\psK}=+\infty.
\]
Let $j_{c}=\min\{0<\ell\leq j_0:\limsup_{K\to+\infty} \rho_{\ell}^{\psK}=+\infty\}$.
Hence there exists $\alpha<+\infty$ such that $\limsup_{K\to+\infty} \rho_{j_c-1}^{\psK}=\alpha$.
By definition of $j_c$, there exists a diverging sequence $(K_p)_p$ such that $\lim_{p\to+\infty} \rho_{j_{c}}^{\psKp}=+\infty$.
Let $\rho\in S_1\cup S_2$ such that $\rho>\alpha$. 

If $j\leq j_c-1$, we have $\limsup_{p\to+\infty} \rho_{j}^{\psKp}\leq \limsup_{p\to+\infty} \rho_{j_c-1}^{\psKp}\leq \alpha<\rho$.
For all $j\geq j_c$, we have $\liminf_{p\to+\infty} \rho_{j}^{\psKp}\geq \liminf_{p\to+\infty} \rho_{j_c}^{\psKp}=+\infty$.
This implies $\rho\notin G$, contradicting Theorem \ref{liminf}.
\end{proof}

\begin{theorem}\label{prsquelafin}
We have
\[
S_1\cup S_2\supset G.
\]
Moreover, for each $j\in\entiers$, let $(K_{p})_p$ be a diverging sequence such that 
\[
\lim_{p\to\infty}\rhoKp{j}{p}=\rho_{*}
\]
where $\rho_*$ is finite by Corollary \ref{cfini}.
Then
\begin{enumerate}[1)]
\item
If $\rho_{*}\in S_{1}\Delta S_{2}$ then
\begin{equation}\label{difsym}
\liminf_{p\to\infty}\; \min\left\{\bigg|\rhoKp{j+1}{p}-\rho_{*}\bigg|,\; \bigg|\rhoKp{j-1}{p}-\rho_{*}\bigg|\right\}>0\;.
\end{equation}
Moreover there are only two cases:
\begin{enumerate}
\item
If $\rho_*\in S_1$ then there exists a diverging sequence of integers $(p_\ell)$ such that $Q_{{\scriptscriptstyle K_{p_\ell}}} \phi^{{\scriptscriptstyle(K_{p_{\ell}}})}_j \xrightarrow[]{\Ldeux} \varphi_*$,
where $\rho_*$ and $\varphi_*$ are such that $\HO \varphi_*= - \rho_* \varphi_*$.
\item
If $\rho_*\in S_2$ then $\phi^{{\scriptscriptstyle(K_{p})}}_j\xrightarrow[]{\ldeux} \phi_*$, where $\rho_*$ and $\phi_*$ are such that $\PO \phi_*= - \rho_* \phi_*$.
\end{enumerate}
\item If $\rho_{*}\in S_{1}\cap S_{2}$ then we have the following two assertions:
\begin{enumerate}
\item
There exists a diverging sequence of integers $(p_{\ell})$ such that:
\[
\text{either}\quad \lim_{\ell\to\infty}\rhoKp{j+1}{p_{\ell}}=\rho_{*}\quad\text{or}\quad\lim_{\ell\to\infty}\rhoKp{j-1}{p_{\ell}}=\rho_{*}.
\]
\item
We have
\begin{equation}\label{madmin}
\liminf_{p\to\infty}\; \min\left\{\bigg|\rhoKp{j+1}{p}-\rho_{*}\bigg|,\; \bigg|\rhoKp{j-1}{p}-\rho_{*}\bigg|\right\}=0
\end{equation}
and
\begin{equation}\label{madmax}
\liminf_{p\to\infty}\; \max\left\{\bigg|\rhoKp{j+1}{p}-\rho_{*}\bigg|,\; \bigg|\rhoKp{j-1}{p}-\rho_{*}\bigg|\right\}>0\;.
\end{equation}
\end{enumerate}
\end{enumerate}
\end{theorem}
Note that \eqref{difsym} means that if $\rho_{*}\in S_{1}\Delta S_{2}$ then $-\rho_{*}$ is a simple asymptotic eigenvalue, and either $-\rho_{*}$
is an eigenvalue of $\HO$ if $\rho_{*}\in  S_{1}$, or of  $\PO$ if $\rho_{*}\in S_{2}$. In addition,  \eqref{madmin} and \eqref{madmax} mean that if
$\rho_{*}\in  S_{1}\cap S_{2}$, then $-\rho_{*}$ is a double asymptotic eigenvalue which is an eigenvalue of both $\HO$ and  $\PO$. 
 
\noindent\textbf{Proof of Theorem \ref{main-theorem}.} The proof is recursive.
For $j=0$ it follows from \cite{CCM1} that $\lim_{\sK\to+\infty} \rho_0^{\psK}=0$.
Let $j\geq 0$ and assume that for $\ell\leq j$ (if any) $\lim_{\sK\to+\infty} \rho_\ell^{\psK}=\eta_\ell$.
We now prove that $\lim_{\sK\to+\infty} \rho_{j+1}^{\psK}=\eta_{j+1}$. There are several cases to consider.
\begin{itemize}
\item
If $\eta_j\in S_1\Delta S_2$, we claim that $\liminf_{\sK\to+\infty}\rho_{j+1}^{\psK}\geq \eta_{j+1}$. Otherwise, by Theorem \ref{prsquelafin} and the recursive hypothesis, 
there would exist $K_p\to+\infty$ such that $\rho_{j+1}^{\psKp}\to \eta_*<\eta_{j+1}$. Since by the recursive hypothesis 
$\lim_{\sK\to+\infty} \rho_j^{\psKp}=\eta_j$, we have $\eta_*\geq \eta_j$. From the first statement of Theorem \ref{prsquelafin} it follows that $\eta_*= \eta_j$.
This contradicts 1) of Theorem \ref{prsquelafin}.\newline
We now claim that $\limsup_{\sK\to+\infty}\rho_{j+1}^{\psK}\leq \eta_{j+1}$. Otherwise, by Theorem \ref{prsquelafin} and the recursive hypothesis, 
there would exist $K_p\to+\infty$ such that $\rho_{j+1}^{\psKp}\to \eta_*>\eta_{j+1}$. This implies that $\eta_{j+1}\not\in G$, contradicting Theorem \ref{liminf}.\newline
Hence, in this case, $\lim_{\sK\to+\infty}\rho_{j+1}^{\psK}= \eta_{j+1}$.
\item
If $\eta_j\in S_1\cap S_2$ (which implies $j>0$), we have two cases:
\begin{itemize}
\item
If $\eta_{j-1}=\eta_j$, then we claim that $\liminf_{\sK\to+\infty}\rho_{j+1}^{\psK}\geq \eta_{j+1}$. Otherwise, by the same argument as before, 
there would exist $K_p\to+\infty$ such that $\rho_{j+1}^{\psKp}\to \eta_{j}$, contradicting \eqref{madmax} in Theorem \ref{prsquelafin}.\newline
We now claim that $\limsup_{\sK\to+\infty}\rho_{j+1}^{\psK}\leq \eta_{j+1}$. Otherwise, as before, this would contradict that $\eta_{j+1}\in G$.
\newline
Hence, in this case, $\lim_{\sK\to+\infty}\rho_{j+1}^{\psK}= \eta_{j+1}$.
\item
If $\eta_{j-1}<\eta_j$, then we obviously have $\liminf_{\sK\to+\infty}\rho_{j+1}^{\psK}\geq \eta_j$. 
If  $\limsup_{\sK\to+\infty}\rho_{j+1}^{\psK}> \eta_j$, then there exists $K_p\to+\infty$ such that $\rho_{j+1}^{\psKp}\to \eta_*>\eta_j$,
contradicting \eqref{madmin} in Theorem \ref{prsquelafin}.\newline
Hence, in this case, $\lim_{\sK\to+\infty}\rho_{j+1}^{\psK}= \eta_{j+1}$.
\end{itemize}
\end{itemize}
Therefore we $\lim_{\sK\to+\infty}\rho_{j+1}^{\psK}= \eta_{j+1}$. As announced, the proof of Theorem \ref{main-theorem} follows recursively.

\section{Properties of the eigenvectors}

Our aim in this part is to prove that for $K$ large enough,  the  eigenvectors $\phi^{\psK}_{j}$ of $\LKl$  (see \eqref{def-eigen-LK}) are functions whose representation is 
sketched in the figure.
An eigenvector is `negligible' outside the union of a neighborhood of $1$, and a neighborhood of $K\xf$. It is `non-negligible' in at least one of these neighborhoods.

\begin{figure}[htb!]
\centering
\includegraphics[scale=.75]{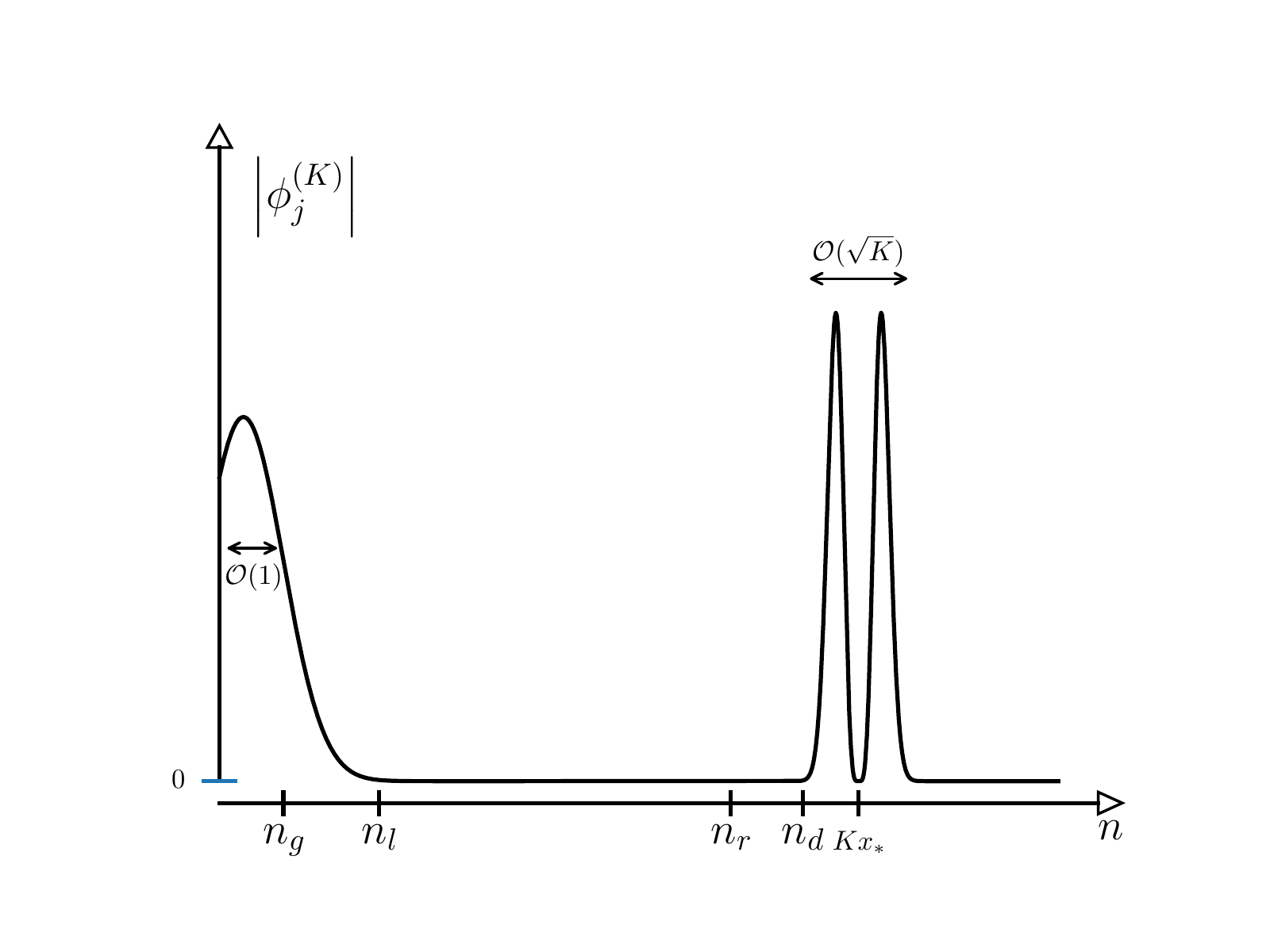}
\legend{{\small Figure: Schematic representation of one of the three possible `shapes' of the eigenvectors $\phi^{\psK}$ (with distortion).}}
\end{figure}

To separate the different behaviors, we introduce a `potential' defined by
\begin{equation}
\label{potential}
\Vn=\lambdaK_{n}+\muK_{n}-\,\sqrt{\lambdaK_{n}\muK_{n+1}}-\,\sqrt{\lambdaK_{n-1}\muK_{n}}\,\un_{\{n>1\}}.
\end{equation}
For $\eta>0$, let $\nge$ and $\nde$ be integers such that $\llbracket\nge,\nde\rrbracket$ is the maximal interval containing $\Kxf/2$ such that
\[
\inf_{n\,\in\, \llbracket\nge,\,\nde\rrbracket}\big(\Vn-\eta\big)>0\;.
\]
Let $\nl=\big\lfloor (\log K)^{2}\big\rfloor$ and $\nr=\big\lfloor\Kxf-K^{\frac{2}{3}}\, \log K\big\rfloor$. It follows from our assumptions that for $K$ large enough
\[
1<\nge<\nl\ll \frac{\Kxf}{2}<\nr<\nde\;.
\]

\begin{proposition}\label{strucvp}
For any $\eta>0$ there exists $a_{\eta}>0$, and $K_{\eta}>0$ such that, if $K>K_{\eta}$ and $\phi$ of norm one in $\ldeux$ satisfies
\[
\LKl\phi=-\rho\,\phi\quad\text{where}\quad\rho<\eta
\]
then
\[
\sup_{\nl\le n\le \nr}\big|\phi(n)\big| \le \e^{-a_{\eta}(\log K)^{2}}\;.
\]
\end{proposition}
\begin{proof} Let us consider an eigenvector $\phi$ of norm one in $\ldeux$ satisfying
$\ 
\LKl\,\phi= - \rho\,\phi$. 

If $\phi({\nge})\neq 0$, define $\tnge=\nge$ and if needed, change the sign of $\phi$ such that $\phi({\tnge})>0$. If $\phi({\nge})=0$, define 
$\tnge=\nge+1$ and if needed, change the sign of $\phi$ such that $\phi({\tnge})>0$ (note that $\phi({\nge})=\phi({\nge+1})=0$ contradicts the normalisation since
$\phi$ solves a second-order recurrence relation). Then, changing the definition of $\nge$  if necessary,  we can assume that $\phi({\nge})> 0$.

Thanks to the local maximum-minimum principle (see Proposition \ref{principe}) we only have four cases.
\begin{enumerate}[1)]
\item
$\phi({\nge+1})\ge \phi({\nge})$ and $\phi({n})$ is increasing on $\llbracket\nge,\nde\rrbracket$.
\item
$\phi({\nge+1})< \phi({\nge})$ and $\phi({n})$ is decreasing and stays nonnegative  on $\llbracket\nge, \,\nde\rrbracket$.
\item $\phi({\nge+1})< \phi({\nge})$ and $\phi({n})$ has a minimum in the interval $\nmin$ in $\llbracket\nge-1, \,\nde-1\rrbracket$ and $\phi({\nmin})\ge0$. Note that 
$\phi({n})$ is  decreasing on $\llbracket\nge, \,\nmin\rrbracket$ and  increasing on $\llbracket\nmin, \,\nde\rrbracket$.
\item
$\phi({n})$ is decreasing on $\llbracket\nge, \,\nde\rrbracket$ and $\phi({\nde})<0$.
\end{enumerate}
We first observe that since $\tlamz>\tmuz$ we have
\begin{align*}
\MoveEqLeft[10] \lim_{K\to\infty}\sup_{n\,\in \big\llbracket \frac{\nl}{2},\,\nl\big\rrbracket}\frac{\sqrt{\lambdaK_{n-1}\,\muK_{n}}}{\lambdaK_{n}+\muK_{n}-\eta-\sqrt{\lambdaK_{n}\muK_{n+1}}} \\
& =\frac{\sqrt{\tlamz\,\tmuz}}{\tlamz+\tmuz-\sqrt{\tlamz\,\tmuz}}<1.
\end{align*}
We also observe that there exists $c>0$ such that for $K$ large enough, and any $n\in\llbracket\nr,\, \nde\rrbracket$ we have 
\[
\frac{\sqrt{\lambdaK_{n}\muK_{n+1}}}
{\lambdaK_{n}+\muK_{n}-\eta-\sqrt{\lambdaK_{n-1}\,\muK_{n}}}\le\frac{1}{1+\frac{c\,(n-\Kxf)^{2}}{K^{2}}}\;.
\]
The result follows by inspecting the monotonicity in the different cases and using the last part of Proposition \eqref{locmax}.
\end{proof}

\begin{theorem}\label{thm-eig}
Let $\phi\in\Dom(\LKl)\subset \ldeux$ of norm $1$, satisfying
\[
\LKl \phi=-\rho\, \phi
\]
for some real $\rho$.
Then there exist $C(\phi)>0$ and an integer $r(\phi)$ such that for all $n\geq r(\phi)$ we have
\[
|\phi(n)| \leq C(\phi) \, 2^{-n}.
\]
Moreover, $(\phi(n))_n$ does not vanish and is of constant sign.
\end{theorem}

\begin{proof}
It follows from our hypothesis that for any $K>1$ there exists an integer $r_0{\scriptstyle (K)}$ such that, for all $n\geq r_0{\scriptstyle (K)}$, we have
\[
0<\frac{\sqrt{\lambdaK_n\muK_{n+1}}}{\lambdaK_n+\muK_{n+1}+\rho-\sqrt{\lambdaK_{n-1}\,\muK_n}}\leq\frac12.
\]
We can assume that $(\phi(n))_n$ is a sequence of real numbers and $\phi(r_0{\scriptstyle (K)})> 0$.
We start by proving that $(\phi(n))_n$ is positive  and decreasing for $n\ge r_0{\scriptstyle (K)}$.
There are only the following four possibilities. 
\begin{enumerate}
\item
$\phi(r_0{\scriptstyle (K)}+1)\geq \phi(r_0{\scriptstyle (K)})$. It follows from Proposition \ref{pouet} that $\phi$ is increasing for
$n\geq r_0{\scriptstyle (K)}$, contradicting that $\phi$ has norm $1$.
\item
$\phi(r_0{\scriptstyle (K)}+1)< \phi(r_0{\scriptstyle (K)})$, and there exists $r'>r_0{\scriptstyle (K)}$ such that $\phi(r')<0$ and $\phi$ decreases on
$\llbracket r_0{\scriptstyle (K)},r'\rrbracket$, and $\phi\geq 0$ on $\llbracket r_0{\scriptstyle (K)},r'-1\rrbracket$.
Then by Proposition \ref{pouet}, $\phi$ is decreasing for $n\geq r'$, contradicting that $\phi$ is normalized.
\item
$\phi(r_0{\scriptstyle (K)}+1)< \phi(r_0{\scriptstyle (K)})$, and there exists $r'>r_0{\scriptstyle (K)}$ such that $\phi(r')<0$ and $\phi$ is not monotonous on
$\llbracket r_0{\scriptstyle (K)},r'\rrbracket$, and $\phi\geq 0$ on $\llbracket r_0{\scriptstyle (K)},r'-1\rrbracket$. 
Then there exists $r''<r'$ such that $\phi$ is decreasing on $\llbracket r_0{\scriptstyle (K)},r''\rrbracket$, and such that $\phi(r''+1)\geq \phi(r'')$. If $ \phi(r'')>0$,  then we are in 
case 1. If $\phi(r''+1)>0$,  it follows from Proposition \ref{pouet} that $\phi$ is increasing for $n\geq r''+1$, contradicting that $\phi$ has norm $1$.  If $\phi(r'''+1)= \phi(r''')=0$ then $\phi$ is the null sequence
as solution of a second order equation, which leads to a contradiction.
\item
$\phi(r_0{\scriptstyle (K)}+1)< \phi(r_0{\scriptstyle (K)})$, $\phi\geq 0$. Suppose that there exists a local minimum at $r'''$ (finite). Then if $\phi(r'''+1)\geq \phi(r''')>0$, we are in case 1.
If $\phi(r'''+1)>\phi(r''')=0$, then it follows from Proposition \ref{pouet} that $\phi$ is increasing for $n\geq r'''+1$, contradicting that $\phi$ has norm $1$.
If $\phi(r'''+1)= \phi(r''')=0$ then $\phi$ is the null sequence, which leads to a contradiction.
\end{enumerate}

Therefore $\phi$ is striclty positive and monotone decreasing. The result then follows by using the last part of Proposition \ref{locmax}.

\end{proof}

Let us now prove two key lemmas. 

\begin{lemma}\label{convergenceS2}
Let $\phi_{\sK}\in \Dom\big(\mathcal{L}_{K}\big)$ be a normalized sequence such that 
\[
\mathcal{L}_{\sK}\phi^{\psK}=-\rho^{\psK}\phi^{\psK}. 
\]
Assume that there exists a diverging sequence $(K_{p})$ such that 
\[
0\leq \lim_{p\to\infty} \rho^{\psKp}=\rho_{*}<+\infty
\]
and 
\[
\limsup_{p\to\infty}\big\|\phi^{\psKp}\un_{\{\cdot\,\le n_l(K_p)\}}\big\|_{\ldeux}>0.
\]
Then $\rho_{*}\in S_{2}$ and there exists a diverging subsequence
$\big(K_{p_{\ell}}\big)$ such that the limit 
\[
\lim_{\ell\to\infty} \phi^{(K_{p_{\ell}})}=\phi_{*}
\]
exists in $\ldeux$, $\|\phi_{*}\|_{\ldeux}>0$ and $\phi_{*}$ is an eigenvector of\, $\PO$ with eigenvalue $-\rho_{*}$.
\end{lemma}
\begin{proof}
Let $(p_{\ell})$ be a diverging sequence of integers such that
\[
\lim_{\ell\to\infty} \big\|\phi^{(K_{p_{\ell}})}\;\un_{\{\cdot\,\le n_l(K)\}}\big\|_{\ldeux}
=\limsup_{p\to\infty}\big\|\phi^{(K_{p})}\;\un_{\{\cdot\,\le n_l(K_p)\}}\big\|_{\ldeux}\;.
\]
We define for each $\ell$ a  normalized sequences in $\ldeux$ by
\[
\psi_{\ell}(n)=\frac{\phi^{(K_{p_{\ell}})}(n)\;  \un_{\{n\,\le n_l(K_{p_{\ell}})\}} }{
\big\|\phi^{(K_{p_{\ell}})}\;\un_{\{\cdot\,\le n_l(K_{p_\ell})\}}\big\|_{\ldeux}}\,.
\]
It is easy to verify using Proposition \ref{strucvp}  that $\psi_{\ell}\in \Dom(\PO)$ and 
\[
\big\|\PO \psi_{\ell}+\rho^{(K_{p_{\ell}})}\psi_{\ell}\|_{\ldeux}\le \Oun\;
\frac{(\log K_{p_{\ell}})^{2} \,\e^{-a\,(\log K_{p_{\ell}})^{2}}}{\big\|\phi^{(K_{p_{\ell}})}\;
\un_{\{\cdot\,\le n_l(K_{p_{\ell}})\}}\big\|_{\ldeux}}\;.
\]
The first result follows from Proposition \ref{yaspectre} since the r.h.s. tends to zero. 

The second result follows from Proposition \ref {approxvp} since 
the spectrum of $\PO$ is discrete and
simple by Theorem \ref{specdeux}.  
\end{proof}

\begin{lemma}\label{convergenceS1}
Let $\phi^{(K)}\in \Dom\big(\mathcal{L}_{K}\big)$ be a normalized sequence such that 
\[
\mathcal{L}_{\sK}\phi^{\psK}=-\rho^{\psK}\;\phi^{\psK}.
\]
Let us assume that there exists a diverging subsequence $(K_{p})$ such that
\[
0\leq \lim_{p\to\infty} \rho^{(K_{p})}=\rho_{*}<+\infty
\]
and 
\[
\lim_{p\to\infty}\big\|\phi^{(K_{p})}\;\un_{\{\cdot\,\ge n_r(K_p)\}}\big\|_{\ldeux}>0. 
\]
Then $\rho_{*}\in S_{1}$ and there exists a diverging subsequence of integers $(p_{\ell})$ such that 
\[
\lim_{\ell\to\infty}\frac{ Q_{K_{p_{\ell}}}\phi^{(K_{p_{\ell}})}}{\big\| Q_{K_{p_{\ell}}}\phi^{(K_{p_{\ell}})}\big\|_{\Ldeux}}=\psi_{*}
\]
exists in $\Ldeux$, $\|\psi_{*}\|_{\Ldeux}=1$, and $\psi_{*}$ is an eigenvector of\, $\HO$ with eigenvalue $-\rho_{*}$.
\end{lemma}
\begin{proof}
We define for each $p$ a  normalized sequence in $\ldeux$ by
\[
\psi^{(K_{p})}(n)=\frac{\phi^{(K_{p})}(n)\;  \un_{\{n\,\ge\, n_r(K_p)\}} }{\big\|\phi^{(K_{p})}\;\un_{\{\cdot\,\ge\, n_r(K_p)\}}\big\|_{\ldeux}}\;.
\]
It is easy to verify using Proposition \ref{strucvp} that $\psi^{(K_{p})}\in  \Dom\big(\mathcal{L}_{K_p})$ and 
\begin{equation}\label{quasierreur}
\big\|\mathcal{L}_{K_{p}} \psi^{(K_{p})} +\rho^{(K_{p})}\psi^{(K_{p})}\|_{\ldeux}\le 
\Oun\;\frac{K_{p} \,\e^{-a\,(\log K_{p})^{2}}}{ \big\|\phi^{(K_{p})}\;\un_{\{\cdot\,\ge\, n_r(K_p)\}}\big\|_{\ldeux}}\;.
\end{equation}
We apply Proposition \ref{sobmodif}, Lemma \ref{unifint}, Lemma \ref{regu}, and Theorem \ref{KFR} to conclude that there  exists a diverging sequence
of integers $(p_{\ell})$ such that the sequence of functions $Q_{K_{p_{\ell}}}\psi^{(K_{p_\ell})}$ converges in $\Ldeux$ to a normalised fonction $\psi_{*}$. 

Let $u\in \mathscr{D}$. We have from \eqref{quasierreur}
\[
\lim_{\ell\to\infty}\big\langle P_{K_{p_{\ell}}}u\,,\,\mathcal{L}_{K_{p_{\ell}}}\, \psi^{(K_{p_{\ell}})}+\rho^{(K_{p_{\ell}})}\,\psi^{(K_{p_{\ell}})}\big\rangle_{\ldeux}=0
\]
hence (since $P_{K_{p_{\ell}}}\,u\in \Dom\big(\mathcal{L}_{K_{p_{\ell}}}\big)$)
\[
\lim_{\ell\to\infty}\big\langle \mathcal{L}_{K_{p_{\ell}}}\,P_{K_{p_{\ell}}}u\,,\,\psi^{(K_{p_{\ell}})}\big\rangle_{\ldeux}=
-\rho_{*}\lim_{\ell\to\infty}\big\langle P_{K_{p_{\ell}}}u\,,\,\psi^{(K_{p_{\ell}})}\big\rangle_{\ldeux}\;.
\]
From the isometric property of $Q$ (see Lemma \ref{isometry}) we get
\begin{align*}
\MoveEqLeft \lim_{\ell\to\infty}\big\langle Q_{K_{p_{\ell}}}\mathcal{L}_{K_{p_{\ell}}}\,P_{K_{p_{\ell}}}u\,,\,Q_{K_{p_{\ell}}}\psi^{(K_{p_{\ell}})}\big\rangle_{\Ldeux}\\
&=-\rho_{*}\lim_{\ell\to\infty}\big\langle Q_{K_{p_{\ell}}} P_{K_{p_{\ell}}}u\,,\,Q_{K_{p_{\ell}}} \psi^{(K_{p_{\ell}})}\big\rangle_{\Ldeux}\;.
\end{align*}
In other words 
\[
\lim_{\ell\to\infty}\big\langle \mathscr{L}_{K_{p_{\ell}}}u,Q_{K_{p_{\ell}}}\psi^{(K_{p_{\ell}})}\big\rangle_{\Ldeux}
=-\rho_{*}\lim_{\ell\to\infty}\big\langle Q_{K_{p_{\ell}}}P_{K_{p_{\ell}}}u,Q_{K_{p_{\ell}}}\psi^{(K_{p_{\ell}})}\big\rangle_{\Ldeux}.
\]
Using the convergence in $\Ldeux$ of $\big(\,Q_{K_{p_{\ell}}}\psi^{(K_{p_{\ell}})}\big)_\ell$ to $\psi_{*}$,
Proposition \ref{convS}, and Lemma  \ref{isomQP}, we get for all $u\in \mathscr{D}$,
\[
\big\langle
\HO u,\psi_{*}\big\rangle_{\Ldeux}=-\rho_{*}\big\langle u\,,\, \psi_{*}\big\rangle_{\Ldeux}\;.
\]
Since $\mathscr{D}$ is dense in the domain of the self-adjoint operator $\HO$,  we conclude that $\psi_{*}$ is an eigenvector of $\HO$. 
\end{proof}

We now state three key propositions.

\begin{proposition}\label{decoupe}
Let $j$ be fixed and let $(K_{p})$ be a diverging sequence such that 
\[
\lim_{p\to\infty} \rhoKp{j}{p}=\rho_{*}\;.
\]
Let $\phiKp{j}{p}$ be a normalized eigenvector of $\mathcal{L}_{\sK_p}$ with eigenvalue $-\rhoKp{j}{p}$.
Then there exists an infinite sequence of integers $(p_\ell)$ such that
\begin{itemize}
\item
$
\phiKp{j}{p_\ell}\,\un_{\{\cdot\,<\,n_l(\sK_{p_\ell})\}}\xrightarrow[]{\ldeux} \phi_* 
$, where $\phi_*$ is either the null sequence or an eigenvector of \,$\PO$ with eigenvalue $-\rho_*$, namely  $\rho_*\in S_{2}$.
\item
$
Q_{K_{p_{\ell}}}\phiKp{j}{p_\ell}\,\un_{\{\cdot\,\ge\, n_r(\sK_{p_\ell})\}} \xrightarrow[]{\Ldeux} \varphi_* 
$, where $\varphi_*$ is either the null function or an eigenvector of \,$\HO$ with eigenvalue $-\rho_*$, namely  $\rho_*\in S_{1}$.
\end{itemize}
Moreover, $\|\phi_*\|_{\ldeux}^2+\|\varphi_*\|_{\Ldeux}^2=1$.
\end{proposition}
\begin{proof}
We have either
\[
\lim_{p\to\infty}\Big\|\phiKp{j}{p}\,\un_{\{\cdot\,<n_l(K_p)\}}\Big\|_{\ldeux}=0
\]
or there exists an infinite sequence of integers $(p_{\ell})$ such that
\[
\lim_{\ell\to\infty}\Big\|\phiKp{j}{p_\ell}\,\un_{\{\cdot\,<n_l(K_{p_{\ell}})\}}\Big\|_{\ldeux}>0\,.
\]
The first statement of the proposition follows from Lemma \ref{convergenceS2} applied to the diverging sequence $(K_{p_\ell})$.
Similarly, either 
\[
\lim_{p\to\infty}\Big\|\phiKp{j}{p}\,\un_{\{\cdot\,\ge\,\nr\}}\Big\|_{\ldeux}=0
\]
or there exists an infinite sequence of integers $(p_{\ell})$ such that
\[
\lim_{\ell\to\infty}\Big\|\phiKp{j}{p_\ell}\,\un_{\{\cdot\,\ge\, n_r(K_{p_\ell})\}}\Big\|_{\ldeux}>0.
\]
The second statement of the proposition follows from Lemma \ref{convergenceS1} applied to the diverging sequence $(K_{p_\ell})$.
The last statement follows from the normalisation $\phiKp{j}{p}$ and Proposition \ref{strucvp}.
\end{proof}

\begin{proposition}\label{prop-quasi-S2}
Let $\rho\in S_2$.
Then  there exists a normalized vector $v_\rho\in\Dom(\mathcal{L}_{\sK})$  with eigenvalue $-\rho$ such that
\[
\lim_{K\to\infty} \|\LKl v_\rho +\rho\, v_\rho\|_{\ell^2}=0.
\]
Moreover
\[
\lim_{K\to+\infty} \big\| v_\rho \un_{\{\cdot\,>\,n_l(\sK)\}}\big\|_{\ldeux}=0.
\]
\end{proposition}
\begin{proof}
Let $v_\rho$ be a normalized eigenvector corresponding to the eigenvalue $-\rho$ for the operator $\PO$ in the space $\ldeux$.
From the assumptions we have for $n\leq \lfloor \log K\rfloor$ 
\[
\lambdaK_n=\tlamz + \mathcal{O}\left(\frac{\log K}{K}\right)\quad\text{and}\quad\muK_n=\tmuz+ \mathcal{O}\left(\frac{\log K}{K}\right).
\]
Since $\PO v_\rho+\rho\, v_\rho=0$, the reader can easily check that
\[
\|\LKl v_\rho +\rho\, v_\rho\|_{\ell^2(\integers\cap \{ 1,\ldots,\lfloor\log K\rfloor\})} \leq \mathcal{O}\left(\frac{\log K}{K}\right).
\]
Now using Theorem \ref{specdeux} we have
\[
v_{\rho}(n)=\sqrt{n}\,\left(\frac{\tmuz}{\tlamz}\right)^{\frac{n}{2}}P(n)
\]
where $P$ is certain polynomial. Hence there exists $c_v>0$ such that for any $n\in \integers$,
$|v_{\rho}(n)|\leq c_v \left(\frac{\tmuz}{\tlamz}\right)^{\frac{n}{4}}$.
By \eqref{hipopo} there exists $c>0$ such $b(x)+d(x)\leq c\, \e^x$ for all $x\geq 0$. 
This implies that $\lambda_n^{\psKp}\leq c K_p\, \e^{n/K_p}$ and $\mu_n^{\psKp}\leq c K_p \,\e^{n/K_p}$.
The reader can easily check that
\[
\|\LKl v_\rho +\rho\, v_\rho\|_{\ell^2(\integers\cap \{\lfloor \log K\rfloor+1,\ldots,\infty\})} \leq \Oun \left(\frac{\tmuz}{\tlamz}\right)^{\frac{\log K}{4}}.
\]
It follows that $v_\rho\in \Dom(\LKl)$, and (remember that $\tlamz>\tmuz$)
\[
\lim_{K\to\infty} \|\LKl v_\rho +\rho \, v_\rho\|_{\ldeux}=0.
\]
The other statement follows at once from the exponential decay of $v$.
\end{proof}

\begin{proposition}\label{prop-quasi-S1}
Let $\rho\in S_1$.
Then  there exists a sequence of normalized vectors $(\psi_\rho^{(K)})_K\subset \ldeux$
such that $\psi_\rho^{\psK}\in\Dom(\mathcal{L}_{\sK})$, and
\[
\lim_{K\to\infty}\big\| \LKl \psi_\rho^{\psK}+\rho\, \psi_\rho^{\psK}\big\|_{\ldeux}=0.
\]
Moreover
\[
\lim_{K\to\infty}\big\|  \psi_\rho^{\psK} \un_{\{\cdot\,<n_r\psK\}}\big\|_{\ldeux}=0.
\]
\end{proposition}
\begin{proof}
Let $\varphi_\rho$ be a real normalized eigenvector of $\HO$ corresponding to the eigenvalue $-\rho$ (in $\Ldeux$).
Since $\varphi_\rho$ is a (rescaled) Hermite function (see \cite{babusci-et-al} or Remark \ref{rem:OUHO}),  it satisfies the hypothesis of Proposition \ref{convS}, hence
\[
\lim_{K\to\infty}\big\|\mathscr{L}_{\sK} \varphi_\rho+\rho\,\varphi_\rho\big\|_{\Ldeux}=0.
\]
Using Lemma \ref{isomQP} we get
\[
\lim_{K\to\infty}\big\|\mathscr{L}_{\sK} \varphi_\rho+\rho\;Q_{\sK}P_{\sK}\varphi_\rho\big\|_{\Ldeux}=0
\]
and then
\[
\lim_{K\to\infty}\big\|\mathcal{L}_{\sK} \, P_{\sK}\varphi_\rho+\rho\,P_{\sK}\varphi_\rho\big\|_{\ldeux}=0\;
\]
and 
\[
\lim_{K\to\infty}\big\|P_{\sK}\varphi_\rho\big\|_{\ldeux}=1.
\]
The first statement follows by letting $\psi_{\rho}^{\psK}(n)=(P_{\sK}\varphi_\rho)(n)$ where $P_{\sK}$ is defined in \eqref{PP}. 
The other statement follows from an exponential bound on the decay of $\varphi_\rho$.
\end{proof}

\section{Proof of Theorems \ref{liminf} and \ref{prsquelafin}}\label{proof-maintheorem}

\subsection{Proof of Theorem \ref{liminf}}\label{proof-liminf}

The proof of Theorem \ref{liminf} is an immediate consequence of the
following two propositions.

\begin{proposition}
We have  $\ S_2\subset  G$.
\end{proposition}
\begin{proof}
The proof is by contradiction.
 Let $\rho\in S_2$ and assume 
$\rho\notin  G$.
Then there exists $\eta>0$ be such that for $K$ large enough, $\left[\,\rho-\eta,\rho+\eta\,\right]\cap  G=\emptyset$.
It follows from Proposition \ref{prop-quasi-S2} that there exists a normalized vector $v$ in $\ldeux$, such that $v\in\Dom(\LKl)$,  and
\[
\lim_{K\to\infty} \|\LKl v +\rho\, v\|_{\ldeux}=0.
\]
Therefore, using Proposition \ref{yaspectre}, we obtain a contradiction, hence the proof is finished.
\end{proof}

\begin{proposition}
We have $\ S_1\subset G$.
\end{proposition}
\begin{proof}
The proof is by contradiction.
Let $\rho\in S_1$ and assume $\rho\notin G$.
Then there exists $\eta>0$ be such that for  $K$ large enough, $[\rho-\eta,\rho+\eta]\cap G=\emptyset$.
It follows from Proposition \ref{prop-quasi-S1} that there exists a sequence of normalized vectors $(\psi^{(K)})_K\subset \ldeux$
such that $\psi^{(K)}\in\Dom(\LKl)$  and
\[
\lim_{K\to\infty}\big\| \LKl \psi^{\psK} +\rho\, \psi^{\psK}\big\|_{\ldeux}=0.
\]
Therefore, using Proposition \ref{yaspectre}, we obtain a contradiction, hence the proof is finished.
\end{proof}

\subsection{Proof of Theorem \ref{prsquelafin}}

The first statement follows at once from Proposition \ref{decoupe}.\\
The proof of \eqref{difsym} is by contradiction. 
Assume 
\[
\liminf_{p\to\infty}\; \min\left\{\bigg|\rhoKp{j+1}{p}-\rho_{*}\bigg|, \bigg|\rhoKp{j-1}{p}-\rho_{*}\bigg|\right\}=0.
\]
Assume $\rho_*\in S_1\backslash S_2$ and let $(p_\ell)$  be a diverging sequence of positive integers such that
$\rho_i^{(\sK_{p_\ell})}\to \rho_*$ and $\rho_{i+1}^{(\sK_{p_\ell})}\to \rho_*$ (where $i=j$ or $i=j-1$).
Let $\phi_i^{(\sK_{p_\ell})}$ be a normalized eigenvector of $\mathcal{L}_{K_{p_{\ell}}}$ corresponding to the eigenvalue $-\rho_i^{(\sK_{p_\ell})}$.
We define $\phi^{(\sK_{p_\ell})}_{i+1}$ similarly.
We claim that
\[
\limsup_{\ell\to+\infty} \Big\|\phiKp{i}{p_\ell}\,\un_{\{\cdot\,<\,n_l(\sK_{p_\ell})\}}\Big\|_{\ldeux}=0.
\]
Otherwise Proposition \ref{decoupe} would imply that $\rho_*\in S_2$, a contradiction.
By Proposition \ref{strucvp} we have
\[
\lim_{\ell\to+\infty} \Big\|\phiKp{i}{p_\ell}\,\un_{\{\cdot\,>\,n_r(\sK_{p_\ell})\}}\Big\|_{\ldeux}=1.
\]
This implies 1-a). \\
By a similar argument, we have 
\[
\lim_{\ell\to+\infty} \Big\|\phiKp{i+1}{p_\ell}\,\un_{\{\cdot\,>\,n_r(\sK_{p_\ell})\}}\Big\|_{\ldeux}=1.
\]
By Proposition \ref{decoupe}, there exists a diverging sequence of integers $(\ell_r)$ such that
\[
\Big\| Q_{K_{p_{\ell_r}}}\Big(\phi_i^{(K_{p_{\ell_r}})}\,\un_{\{\cdot\,>\,n_r(\sK_{p_{\ell_r}})\}}\Big)-\varphi_*\Big\|_{\Ldeux}\to 0
\]
where $\varphi_*\in\Dom(\HO)$ is a normalized eigenfunction of $\HO$ corresponding to the eigenvalue $-\rho_*$.
By Proposition \ref{decoupe} again, there exists a diverging sequence of integers $(r_s)$ such that
\[
\Big\| Q_{K_{p_{\ell_{r_s}}}}\Big(\phi_{i+1}^{(K_{p_{\ell_{r_s}}})}\,\un_{\{\cdot\,>\,n_r(\sK_{p_{\ell_{r_s}}})\}}\Big)-\varphi'_*\Big\|_{\Ldeux}\to 0
\]
where $\varphi'_*\in\Dom(\HO)$ is a normalized eigenfunction of $\HO$ corresponding to the eigenvalue $-\rho_*$.
Since $\phi_{i}^{(K_{p_{\ell_{r_s}}})}$ and $\phi_{i+1}^{(K_{p_{\ell_{r_s}}})}$ are orthogonal in $\ldeux$, it follows from the previous estimates that
\[
\lim_{s\to+\infty}
\left\langle \phi_{i}^{(K_{p_{\ell_{r_s}}})}\,\un_{\{\cdot\,>\,n_r(\sK_{p_{\ell_{r_s}}})\}},\phi_{i+1}^{(K_{p_{\ell_{r_s}}})}\,\un_{\{\cdot\,>\,n_r(\sK_{p_{\ell_{r_s}}})\}}\right\rangle_{\ldeux}=0.
\]
By Lemma \ref{isometry} we have
\[
\left\langle \!Q_{K_{p_{\ell_{r_s}}}}\!\!\Big(\phi_{i}^{(K_{p_{\ell_{r_s}}})}\un_{\{\cdot\,>\,n_r(\sK_{p_{\ell_{r_s}}})\}}\Big),
Q_{K_{p_{\ell_{r_s}}}}\!\!\Big(\phi_{i+1}^{(K_{p_{\ell_{r_s}}})}\un_{\{\cdot\,>\,n_r(\sK_{p_{\ell_{r_s}}})\}}\Big)\!\right\rangle_{\!\!\Ldeux}
\xrightarrow[s\to\infty]{}0.
\]
In other words, $\langle \varphi_*,\varphi'_*\rangle_{\Ldeux}=0$. This
is a contradiction since $\varphi_*$ and $\varphi'_*$ are normalized eigenfunctions of $\HO$ corresponding to the same
eigenvalue $-\rho_*$ which is simple. 

The case $\rho_*\in S_2\backslash S_1$ is similar (using again Proposition \ref{decoupe}), so it is left to the reader.

\smallskip
Let us now assume that $\rho_* \in S_{1}\cap S_{2}$. We now prove \eqref{madmin} by contradiction. So we assume that there exists $\delta>0$ such that
\[
\liminf_{p\to\infty}\; \min\left\{\bigg|\rhoKp{j+1}{p}-\rho_{*}\bigg|, \bigg|\rhoKp{j-1}{p}-\rho_{*}\bigg|\right\}>\delta.
\]
Since $\rho_* \in S_{1}$ and from Proposition \ref{prop-quasi-S1}, there exists a sequence of normalized vectors $(\psi_{\rho_*}^{(K_p)})_p\subset \ldeux$
such that $\psi_{\rho_*}^{(K_p)}\in\Dom(\LKlp)$ for all $p$, and
\[
\lim_{p\to\infty}\Big\| \LKlp \psi_{\rho_*}^{(K_p)} -\rho_* \psi_{\rho_*}^{(K_p)}\Big\|_{\ldeux}=0.
\]
We also have (since $\rho^{(\sK_p)}_j \to \rho_*$) that
\[
\lim_{p\to\infty}\Big\| \LKlp \psi_{\rho_*}^{(K_p)} -\rho^{(\sK_p)}_j \psi_{\rho_*}^{(K_p)}\Big\|_{\ldeux}=0.
\]
For each $p$, let $\phi^{(\sK_p)}_j$ be a normalized eigenvector of $\LKlp$ corresponding to the eigenvalue $-\rho^{(\sK_p)}_j$. By using Lemma \ref{approxvp} with $A=\LKlp$, we deduce that there exists a sequence of real numbers $(\theta_p)_p$ such that 
\begin{equation}
\label{pierre1}
\lim_{p\to\infty}\Big\|  \psi_{\rho_*}^{(K_p)} -\e^{\ic \theta_p}\phi^{(\sK_p)}_j\Big\|_{\ldeux}=0.
\end{equation}
Since $\rho_* \in S_{2}$ and from Proposition \ref{prop-quasi-S2}, there exists a normalized vector $v_{\rho_*}\in\Dom(\mathcal{L}_{\sK_p})$ for all $p$  such that
\[
\lim_{p\to\infty} \|\LKlp v_{\rho_*} +\rho_* v_{\rho_*}\|_{\ldeux}=0.
\]
We also have
\[
\lim_{p\to\infty}\Big\| \LKlp v_{\rho_*} +\rho^{(\sK_p)}_j v_{\rho_*}\Big\|_{\ldeux}=0.
\]
By using Lemma \ref{approxvp} with $A=\LKlp$, we deduce that there exists a sequence of real numbers $(\theta'_p)_p$ such that 
\begin{equation}
\label{pierre2}
\lim_{p\to\infty}\Big\| v_{\rho_*} -\e^{\ic \theta'_p} \phi^{(\sK_p)}_j\Big\|_{\ldeux}=0.
\end{equation}
Moreover, from Propositions \ref{prop-quasi-S1} and \ref{prop-quasi-S2}, it follows that
\[
\lim_{p\to\infty}\Big\|  \psi_{\rho_*}^{(K_p)} \un_{\{\cdot\,<\,n_r(\sK_p)\}}\Big\|_{\ldeux}=0\quad\text{and}\quad
\lim_{K\to+\infty} \big\| v_{\rho_*} \un_{\{\cdot\,>\,n_l(\sK)\}}\big\|_{\ldeux}=0
\]
which implies  (using Proposition \ref{strucvp}) that 
\[
\lim_{p\to\infty}\langle  \psi_{\rho_*}^{(K_p)},v_{\rho_*}\rangle=0.
\]
This is a contradiction  with \eqref{pierre1} and  \eqref{pierre2}.
Finally, we prove \eqref{madmax} by contradiction. So we assume that there exists a diverging sequence of integers $(p_\ell)$ such that 
\[
\lim_{\ell\to+\infty}\Big|\,\rhoKp{j+1}{p_\ell}-\rho_{*}\Big|+ \Big|\,\rhoKp{j-1}{p_\ell}-\rho_{*}\Big|=0.
\]
We now apply Proposition \ref{decoupe} with $j$ and $K_{p_\ell}$.
Hence there exists a diverging sequence of integers $(\ell_s)$  such that 
\[
\phiKp{j}{p_{\ell_s}}\,\un_{\{\cdot\,<\,n_l(\sK_{p_{\ell_s}})\}}\xrightarrow[]{\ldeux} \phi_{*,1}
\]
and
\[
Q_{K_{p_{\ell_s}}}\phiKp{j}{p_{\ell_s}}\,\un_{\{\cdot\,\ge\, n_r(\sK_{p_{\ell_s}})\}} \xrightarrow[]{\Ldeux} \varphi_{*,1} 
\]
and we have $\|\phi_{*,1}\|_{\ldeux}^2+\|\varphi_{*,1}\|_{\Ldeux}^2=1$.

We now apply Proposition \ref{decoupe} with $j-1$ and $K_{p_{\ell_s}}$.
Hence there exists a diverging sequence of integers $(s_r)$  such that 
\[
\phiKp{j-1}{p_{\ell_{s_r}}}\,\un_{\{\cdot\,<\,n_l(\sK_{p_{\ell_{s_r}}})\}}\xrightarrow[]{\ldeux} \phi_{*,2}
\]
and
\[
Q_{K_{p_{\ell_{s_r}}}}\phiKp{j-1}{p_{\ell_{s_r}}}\,\un_{\{\cdot\,\ge\, n_r(\sK_{p_{\ell_{s_r}}})\}} \xrightarrow[]{\Ldeux} \varphi_{*,2} 
\]
and we have $\|\phi_{*,2}\|_{\ldeux}^2+\|\varphi_{*,2}\|_{\Ldeux}^2=1$.

We now apply Proposition \ref{decoupe} with $j+1$ and $K_{p_{\ell_{s_r}}}$.
Hence there exists a diverging sequence of integers $(r_q)$  such that 
\[
\phiKp{j+1}{p_{\ell_{s_{r_q}}}}\,\un_{\{\cdot\,<\,n_l(\sK_{p_{\ell_{s_{r_q}}}})\}}\xrightarrow[]{\ldeux} \phi_{*,3}
\]
and
\[
Q_{K_{p_{\ell_{s_{r_q}}}}}\phiKp{j+1}{p_{\ell_{s_{r_q}}}}\,\un_{\{\cdot\,\ge\, n_r(\sK_{p_{\ell_{s_{r_q}}}})\}} \xrightarrow[]{\Ldeux} \varphi_{*,3} 
\]
and we have $\|\phi_{*,3}\|_{\ldeux}^2+\|\varphi_{*,3}\|_{\Ldeux}^2=1$.

Moreover, since $\phiKp{j-1}{p_{\ell_{s_{r_q}}}}, \phiKp{j}{p_{\ell_{s_{r_q}}}}, \phiKp{j+1}{p_{\ell_{s_{r_q}}}}$ are pairwise orthogonal, we have
\begin{equation}\label{lestrois}
\left\langle \phi_{*,m}, \phi_{*,m'}\right\rangle_{\ldeux}
+
\left\langle \varphi_{*,m}, \varphi_{*,m'}\right\rangle_{\Ldeux}
=0,\quad\forall m\neq m'.
\end{equation}
The linear subspace of $\ldeux$ spanned by $\phi_{*,1}, \phi_{*,2}, \phi_{*,3}$ is of dimension at most one because they are eigenvectors of 
$\PO$ for the same simple eigenvalue $-\rho_*$.
The linear subspace of $\Ldeux$ spanned by $\varphi_{*,1}, \varphi_{*,2}, \varphi_{*,3}$ is of dimension at most one because they are eigenfunctions of 
$\HO$ for the same simple eigenvalue $-\rho_*$.
Therefore the subspace of $\ldeux\oplus \Ldeux$ spanned by the three vectors  $(\phi_{*,1},\varphi_{*,1})$,  $(\phi_{*,2},\varphi_{*,2})$,
$(\phi_{*,3},\varphi_{*,3})$ is of dimension at most two. However, these three vectors are normalized and pairwise orthogonal by \eqref{lestrois}.
We thus arrive at a contradiction.


\section{Fr\'echet-Kolmogorov-Riesz compactness criterion and Dirichlet form}\label{auxiliaires}

\subsection{Fr\'echet-Kolmogorov-Riesz compactness criterion}

We recall the  Fr\'echet-Kolmogorov-Riesz compactness criterion in $\Ldeux$.

\begin{theorem}\label{KFR}
Let $(f_{p})_p$ be a normalized sequence in $\Ldeux$ such that the following two conditions are satisfied.
\begin{enumerate}[(i)]
\item
There exists $\epsilon_0>0$ such that there exists a function $R(\epsilon)>0$ on $(0,\epsilon_0)$ such that 
\[
\sup_{p}\int_{\{|x|>R(\epsilon)\}}\big|f_{p}(x)\big|^{2}\dd x\le \epsilon.
\]
\item
There exits a positive function $\alpha$ on $]0,1]$  satisfying 
\[
\lim_{y\searrow\, 0}\alpha(y)=0
\]
and such that for any $p$ and $y\in \left]-1,1\right]$
\[
\int\big|f_{p}(x+y)-f_{p}(x)\big|^{2}\dd x\le \alpha(|y|).
\]
\end{enumerate}
Then one can extract from $(f_{p})_p$ a convergent subsequence  in $\Ldeux$.
\end{theorem} We refer to Remark \textbf{5} on page 387 in \cite{HH}. The following
Lemmas provide expressions for $R(\epsilon)$ and $\alpha(y)$ in our case. Recall that the potential $\Vn$ and $n_{r}(K)$ have been defined in  Section 4 (see \eqref{potential}).

\begin{lemma}\label{unifint}
Let $C>0$.
Let $\mathscr{F}_{C,\,K}$ be the set of normalized sequences $(\phi(n))_n$ in $\ldeux$ such that $\phi(n)=0$ for any $n<n_{r}(K)$ and 
\[
\sum_{n=n_{r}(K)}^{\infty}\big(1\vee \Vn\big)\big)\,\phi^{2}(n)\le C\;.
\]
Then, for any $\phi\in\mathscr{F}_{C,K} $, the function $Q_{\sK}\phi(x)$ satisfies condition \textup{(i)} in Theorem \ref{KFR} with 
\[
R(\epsilon)=\frac{C}{\epsilon}
\]
for any $0<\epsilon<1$.
\end{lemma}

\begin{proof}
We are going to prove that there exist $R_{0}>1$ and  $K_{0}>4$ such that for any $K>K_{0}$, any $R>R_{0}$ and any $\phi\in \mathscr{F}_{C,K}$
\[
\int_{\{|x|>R\}}\big(Q_{\sK}\phi(x)\big)^{2}\dd x\le \frac{C}{R}.
\]
We observe that
\[
\int_{\{|x|>R\}}\big(Q_{\sK}\phi(x)\big)^{2}\dd x\le \sum_{n:|n-\Kxf|>R\,\sqrt{K}-1}\phi^{2}(n).
\]
Our aim is to prove that the right-hand side of the above inequality is bounded above by $C/R$. 

It follows from the hypotheses on $\lambdaK_{n}$ and
$\muK_{n}$  that there exist constants $K_{0}>4$, $1>C_{2}>0$, $\Gamma>0$ and $\zeta>0$, such that for any $K>K_{0}$ there exists an integer
$\Gamma\, K>\ms>2\,\Kxf$ (hence of order $K$) such that  $\muK_{n}>\zeta\,n$ for any $n\ge\ms$ and 
\begin{align*}
\sum_{n=n_{r}(k)}^{\ms}(1\vee\Vn)\;\phi^{2}(n)
&\ge C_{2}\, K^{-1}\sum_{n=n_{r}(K)}^{\ms}(n-\Kxf)^{2} \phi^{2}(n)\\
\sum_{n=\ms+1}^{\infty}(1\vee\Vn)\;\phi^{2}(n)
& \ge C_{2}\, \sum_{n=\ms+1}^{\infty}n\,\phi^{2}(n)\;.
\end{align*}
These estimates imply the following bounds for any integer $L>0$ 
\begin{align}
\label{zone2}
\sum_{\{n_{r}(K)\leq n\leq \ms\}\cap \{|n-\Kxf|>L\}}\phi^{2}(n) & \le \frac{K C}{C_{2}\, L^{2}}\, \un_{\{L<\ms\}}\\
\sum_{\{n>\ms\}\cap \{|n-\Kxf|>L\}}\phi^{2}(n) & \le \frac{C}{C_{2}(L\vee \ms)}
\label{zone3}\;.
\end{align}
We now replace $L$ with $R\,\sqrt K-1$ in the above estimates.

Let $K>4$ be fixed. We distinguish two cases according to the value of $R$. 
\begin{enumerate}[1)]
\item  
$1\le R<\ms/\sqrt K$. Then $\sqrt{K}-1\le L<\ms\le \Gamma K$. Since $L=R\sqrt{K}-1>R\sqrt{K}/2$ (because $K>4$), we have
\[
\sum_{\{|n-\Kxf|>R\,\sqrt{K}-1\}}\phi^{2}(n)\le \frac{4\,\,C}{R^{2}}+\frac{C}{C_{2}\ms}
\le  \frac{4C}{C_{2}R^{2}}+ \frac{C\, \Gamma}{C_{2}R^{2}}\le \frac{C}{R}
\]
if $R>C_{2}^{-1}\big(4+\,\Gamma)$.
\item $R\ge \ms/\sqrt K$. Then $L\ge \ms$. We get
\[
\sum_{\{|n-\Kxf|>R\,\sqrt{K}-1\}}\phi^{2}(n)\le
\frac{C}{C_{2}L}\le \frac{2C}{R\,C_{2}\sqrt{K}}\le \frac{C}{R}
\]
if $C_{2}^{2} K>4$.
\end{enumerate}
We define $K_{0}=5+4\,C_{2}^{-2}$ and $R_{0}=1+C_{2}^{-1}\big(4+\,\Gamma)$.
The result follows.
\end{proof}

\begin{lemma}\label{regu}
Let $(\phi^{{\scriptscriptstyle (K)}})_K$ be a sequence of normalized elements of $\ldeux$ such that $\phiK{(n)}=0$ for $n\le n_{r}(K)$.
Assume also that there exists $C>0$ such that 
\begin{align*}
\MoveEqLeft[10] 
\sup_{K>1} \left\{\sum_{n=n_{r}\psK}^{\infty}\sqrt{\lambdaK_{n}\muK_{n+1}} \big(\phiK{(n+1)}-\phiK{(n)}\big)^{2} \right.\\
& \left. +\sum_{n=n_{r}\psK}^{\infty}(1\vee\Vn)\,\big(\phiK{(n)}\big)^{2}\right\}
\le C.
\end{align*}
Then there exits a positive constant $\tilde C$ such that for any $|h|\le 1$ and any $K>1$
\[
\int \big[Q_{\sK}\phiK{}(x+h)-Q_{\sK}\phiK{}(x)\big]^{2}\dd x\le\alpha(h)=\tilde C\, |h|\,. 
\]
Hence,  for any $(\phiK{})_K$ satisfying the above assumptions, the sequence of functions $(Q_{\sK}\phiK{})_K$ satisfies condition \textit{ii)} in Theorem \ref{KFR} with 
\[
\alpha(y)=\tilde C y\;.
\]
\end{lemma}
\begin{proof}
It is enough to consider the case $0<h<1$.
We first consider the case $0<h\le 1/\sqrt{K}$.
We have
\[
\int\! \big[Q_{\sK}\phiK{}(x+h)-Q_{\sK}\phiK{}(x)\big]^{2}\dd x=
\sum_{q\geq 1}\int_{I^{(K)}_{q}} \!\big[Q_{\sK}\phiK{}(x+h)-Q_{\sK}\phiK{}(x)\big]^{2}\dd x.
\]
Since
\[
Q_{\sK}\phiK{}(x)=K^{\frac{1}{4}}\,\sum_{q\geq 1}\phiK{(q)} \,\un_{I^{\psK}_{q}}(x)
\]
and since the intervals $I^{\psK}_{q}$ are disjoint, we get  
\begin{align*}
\MoveEqLeft \int \big[Q_{\sK}\phiK{}(x+h)-Q_{\sK}\phiK{}(x)\big]^{2}\dd x\\
& =K^{\frac{1}{2}} \sum_{q\geq 1}\int_{I^{\psK}_{q}}
\Big(\phiK{(q)}\,\un_{I^{(K)}_{q}}(x+h) \\
& \hspace{2.8cm}+\phiK{(q+1)}\,\un_{I^{\psK}_{q+1}}(x+h)-\phiK{(q)}\,\un_{I_{q}^{\psK}}(x)\Big)^{2}\dd x\\
& =K^{\frac{1}{2}} \sum_{q\geq 1}\int_{I^{\psK}_{q}} \Big(\big(\phiK{(q)}\big)^{2}\,\un_{I^{\psK}_{q}}(x+h)+\big(\phiK{(q+1)}\big)^{2}\,\un_{I^{\psK}_{q+1}}(x+h)\\
& \hspace{2.8cm}+\big(\phiK{(q)}\big)^{2}\,\un_{I^{\psK}_{q}}(x)
-2\, \big(\phiK{(q)}\big)^{2}\,\un_{I_{q}}(x)\un_{I^{\psK}_{q}}(x+h)\\
& \hspace{2.8cm} -2\, \phiK{(q)}\,\phiK{(q+1)}\,\un_{I^{\psK}_{q}}(x)\un_{I^{\psK}_{q+1}}(x+h)\Big)\dd x
\end{align*}
Let us consider each term separately. Since $\int_{I^{(K)}_{q}} \dd x = K^{-\frac{1}{2}}$, we have
\[
K^{\frac{1}{2}} \sum_{q\geq 1}\int_{I_{q}^{\psK}} \big(\phiK{(q)}\big)^{2}\,\dd x=\sum_{q\geq 1}\,\big(\phiK{(q)}\big)^{2}.
\]
Then we have
\begin{align*}
\MoveEqLeft  K^{\frac{1}{2}}\sum_{q\geq 1}\int_{I_{q}^{\psK}}\big(\phiK{(q)}\big)^{2}\,\un_{I_{q}^{\psK}}(x+h)\,\dd x\\
&=K^{\frac{1}{2}}\sum_{q\geq 1}\,\big(\phiK{(q)}\big)^{2}\int_{\frac{q}{\sqrt{K}}-\frac{1}{2\sqrt{K}}-x_{*}}^{\frac{q}{\sqrt{K}}+\frac{1}{2\sqrt{K}}-x_{*}-h}\,\dd x\\
&=\big(1-h K^{\frac{1}{2}}\big)\,\sum_{q\geq 1}\big(\phiK{(q)}\big)^{2}.
\end{align*}
We also have
\begin{align*}
\MoveEqLeft K^{\frac{1}{2}}\sum_{q\geq 1}\int_{I_{q}^{\psK}} \big(\phiK{(q+1)}\big)^{2}\,\un_{I_{q+1}^{\psK}}(x+h)\,\dd x\\
&=K^{\frac{1}{2}}\sum_{q\geq 1} \big(\phiK{(q+1)}\big)^{2}\int_{\frac{q}{\sqrt{K}}+\frac{1}{2\sqrt{K}}-x_{*}-h}^{\frac{q}{\sqrt{K}}+\frac{1}{2\sqrt{K}}-x_{*}} \dd x\\
&=K^{\frac{1}{2}}\sum_{q\geq 1} \big(\phiK{(q+1)}\big)^{2}\,h\\
&=K^{\frac{1}{2}}\,h\,\sum_{q\geq 1} \big(\phiK{(q)}\big)^{2}.
\end{align*}
Similarly
\begin{align*}
\MoveEqLeft
-2\, K^{\frac{1}{2}} \sum_{q\geq 1}\int_{I_{q}^{\psK}}\big(\phiK{(q)}\big)^{2}\,\un_{I_{q}^{\psK}}(x)\un_{I_{q}}(x+h)\, \dd x\\
& =-2\,\big(1-h K^{\frac{1}{2}}\big)\,\sum_{q\geq 1} \big(\phiK{(q)}\big)^{2}
\end{align*}
and
\begin{align*}
\MoveEqLeft -2 K^{\frac{1}{2}} \sum_{q\geq 1} \int_{I_{q}^{\psK}} \phiK{(q)}\,\phiK{(q+1)}\,\un_{I_{q}^{\psK}}(x)\un_{I_{q+1}^{\psK}}(x+h) \,\dd x\\
&=-2 h K^{\frac{1}{2}} \sum_{q\geq 1}\,\phiK{(q)}\,\phiK{(q+1)}\,.
\end{align*}
We rewrite the last term:
\begin{align*}
\MoveEqLeft -2\,h\,\sqrt{K}\, \sum_{q\geq 1}\,\phiK{(q)}\,\phiK{(q+1)}\\
& =-2\,h\,\sqrt{K}\, \sum_{q\geq 1}\,\big(\phiK{(q)}\big)^{2}+2\,h\,\sqrt{K}\, \sum_{q\geq 1}\,\phiK{(q)}\,\big(\phiK{(q+1)}-\phiK{(q)}\big).
\end{align*}
Summing up, we get
\begin{align*}
\MoveEqLeft  \int \big[Q_{\sK}\phiK{}(x+h)-Q_{\sK}\phiK{}(x)\big]^{2}\dd x\\
& = \sum_{q\geq 1}\,\big(\phiK{(q)}\big)^{2}+\big(1-h\,\sqrt{K}\big)\,\sum_{q\geq 1}\big(\phiK{(q)}\big)^{2}+K^{\frac{1}{2}}\,h\,\sum_{q\geq 1} \big(\phiK{(q)}\big)^{2}\\
& \quad - 2\,\big(1-h\,\sqrt{K}\big)\,\sum_{q\geq 1} \big(\phiK{(q)}\big)^{2}-2\,h\,\sqrt{K}\, \sum_{q\geq 1}\,\big(\phiK{(q)}\big)^{2}\\
& \quad + 2\,h\,\sqrt{K}\, \sum_{q\geq 1}\,\phiK{(q)}\,\big(\phiK{(q+1)}-\phiK{(q)}\big)\\
& = 2\,h\,\sqrt{K}\, \sum_{q\geq 1}\,\phiK{(q)}\,\big(\phiK{(q+1)}-\phiK{(q)}\big).
\end{align*}
By Cauchy-Schwarz inequality we get 
\begin{align*}
\MoveEqLeft \int \big[Q_{\sK}\phiK{}(x+h)-Q_{\sK}\phiK{}(x)\big]^{2}\dd x\\
& \le 2 h\sqrt{K}\,\left(\sum_{q\geq 1}\,\big(\phiK{(q+1)}-\phiK{(q)}\big)^{2}\right)^{\frac{1}{2}}\big\|\phiK{}\big\|_{\ldeux}.
\end{align*}
From the assumption of the lemma and using 
\[
\inf_{n\,\ge \big\lfloor \frac{K\xf}{3}\big\rfloor-1}\sqrt{\lambdaK_{n}\muK_{n+1}}> \zeta K
\]
for some $\zeta>0$ independent of $K$, we get
\[
K^{\frac{1}{2}}\left(\sum_{q\geq 1}\,\big(\phiK{(q+1)}-\phiK{(q)}\big)^{2}\right)^{1/2}\le \sqrt{{C\over \zeta}}.
\]
Therefore, since the sequence $\phiK{}$ is normalized, we get
\[
\int \big[Q_{\sK}\phiK{}(x+h)-Q_{\sK}\phiK{}(x)\big]^{2}\dd x\le 2 h\; \sqrt{{C\over \zeta}}.
\]
We now consider the case $1>h>1/\sqrt{K}$. Let $r=\lfloor h\sqrt{K}\rfloor$ and $h'=h-r/\sqrt{K}$. Note that $0\le h'\le 1/\sqrt{K}$.
We have
\begin{align*}
\MoveEqLeft \int \big[Q_{\sK}\phiK{}(x+h)-Q_{\sK}\phiK{}(x)\big]^{2}\dd x\\
& =\int\bigg[\sum_{j=1}^{r-1}\left(Q_{\sK}\phiK{}\Big(x+\frac{j+1}{\sqrt{K}}\Big)-Q_{\sK}\phiK{}\Big(x+\frac{j}{\sqrt{K}}\Big)\right)\\
&\hspace{1.2cm} +Q_{\sK}\phiK{}\Big(x+\frac{r}{\sqrt{K}}+h'\Big)-Q_{\sK}\phiK{}\Big(x+\frac{r}{\sqrt{K}}\Big)\bigg]^{2}\dd x\\
& \le 2 \int \bigg[\sum_{j=1}^{r-1}\left(Q_{\sK}\phiK{}\Big(x+\frac{j+1}{\sqrt{K}}\Big)-Q_{\sK}\phiK{}\Big(x+\frac{j}{\sqrt{K}}\Big)\right)\bigg]^{2} \dd x\\
&\hspace{1.4cm} + 2 \int \bigg[Q_{\sK}\phiK{}\Big(x+\frac{r}{\sqrt{K}}+h'\Big)-Q_{\sK}\phiK{}\Big(x+\frac{r}{\sqrt{K}}\Big)\bigg]^{2}\dd x.
\end{align*}
We have already estimated the last term. For the first term we now observe that $e^{\psK}_{n}\Big(x+\frac{j}{\sqrt{K}}\Big)=e^{\psK}_{n-j}(x)$. Therefore we can write 
\begin{align*}
\MoveEqLeft\sum_{j=1}^{r-1}\left(Q_{\sK}\phiK{}\Big(x+\frac{j+1}{\sqrt{K}}\Big)-Q_{\sK}\phiK{}\Big(x+\frac{j}{\sqrt{K}}\Big)\right) \\
&=\sum_{j=1}^{r-1}\left(\sum_{n}\phiK{(n)}\, e^{\psK}_{n-j-1}(x)-\sum_{n}\phiK{(n)}\, e^{\psK}_{n-j}(x)\right)\\
&=\sum_{j=1}^{r-1}\left(\sum_{p}e^{\psK}_{p}(x)\,\big(\phiK{(p+j+1)}-\phiK{(p+j)}\right)\\
&=\sum_{p}e^{\psK}_{p}(x)\,\sum_{j=1}^{r-1}\big(\phiK{(p+j+1)}-\phiK{(p+j)}\big).
\end{align*}
This implies
\begin{align*}
\MoveEqLeft \int \bigg[\sum_{j=1}^{r-1}Q_{\sK}\phiK{}\Big(x+\frac{j+1}{\sqrt{K}}\Big)-Q_{\sK}\phiK{}\Big(x+\frac{j}{\sqrt{K}}\Big)\bigg]^{2} \dd x\\
& =\sum_{p}\left(\sum_{j=1}^{r-1}\big(\phiK{(p+j+1)}-\phiK{(p+j)}\big)\right)^{2}\\
& \le r\, \sum_{p}\sum_{j=1}^{r-1}\big(\phiK{(p+j+1)}-\phiK{(p+j)}\big)^{2}\\
&=r^{2}\, \sum_{p}\big(\phiK{(p+1)}-\phiK{(p)}\big)^{2}\le \frac{C }{\zeta } \;\frac{r^{2}}{K}
\end{align*}
as we have seen before. We observe that $r^{2}/K\le h^{2}\le h$ since $h\le 1$, and the result follows by taking  $\tilde C= 2 \sqrt{\frac{C}{\zeta } } + \frac{C}{\zeta }$.
\end{proof}

\subsection{Dirichlet form for the operator \texorpdfstring{$\LKl$}{LKI}}

We need an estimate on the decay at infinity of the eigenfunctions. Note that since the eigenvalues are real, we can assume that the eigenfunctions are real.

\begin{proposition}\label{diri}
If $\phi$ is a normalized sequence  in $\Dom(\LKl)$ decaying exponentially fast at infinity, then   
\begin{align*}
&-\langle\phi,\LKl\phi\rangle \\
& =\sum_{n=1}^{\infty}\sqrt{\lambdaK_{n}\muK_{n+1}}\,
\big(\phi(n+1)-\phi(n)\big)^{2}
+\sum_{n=1}^{\infty}\Vn\,{\phi(n)}^{2}-\frac{1}{2}\sqrt{\lambdaK_{1}\muK_{2}}\,\phi(1)^{2}.
\end{align*}
\end{proposition}
\begin{proof}
For any fixed positive integer $N$ we have
\begin{align*}
\MoveEqLeft \sum_{n=1}^{N}\phi(n)\,(\LKl\phi)(n)\\
&= \sum_{n=1}^{N}\sqrt{\lambdaK_{n}\muK_{n+1}}\,\phi(n)\,\phi(n+1)
+\sum_{n=2}^{N}\sqrt{\lambdaK_{n-1}\,\muK_{n}}\,\phi(n)\,\phi(n-1)\\
&\quad  -\sum_{n=1}^{N}\big(\lambdaK_{n}+\muK_{n}\big)\,\phi(n)^{2}\\
& =-\frac{1}{2}\sum_{n=1}^{N}\sqrt{\lambdaK_{n}\muK_{n+1}}\,(\phi(n)-\phi(n+1))^{2}
+\frac{1}{2}\sum_{n=1}^{N}\sqrt{\lambdaK_{n}\muK_{n+1}}\,\phi(n)^{2}\\
& \quad +\frac{1}{2}\sum_{n=1}^{N}\sqrt{\lambdaK_{n}\muK_{n+1}}\,\phi(n+1)^{2}-\frac{1}{2}\sum_{n=2}^{N}\sqrt{\lambdaK_{n-1}\,\muK_{n}}\,(\phi(n)-\phi(n-1))^{2}\\
& \quad +\frac{1}{2}\sum_{n=2}^{N}\sqrt{\lambdaK_{n-1}\,\muK_{n}}\,\phi(n)^{2}
+\frac{1}{2}\sum_{n=2}^{N}\sqrt{\lambdaK_{n-1}\,\muK_{n}}\,\phi(n-1)^{2}\\
& \quad -\sum_{n=1}^{N}\big(\lambdaK_{n}+\muK_{n}\big)\,\phi(n)^{2}\\
& =-\sum_{n=1}^{N-1}\sqrt{\lambdaK_{n}\muK_{n+1}}\,(\phi(n)-\phi(n+1))^{2}\\
& \quad -\sum_{n=1}^{N}\left(\lambdaK_{n}+\muK_{n}-\, \sqrt{\lambdaK_{n}\muK_{n+1}}-\, \sqrt{\lambdaK_{n-1}\,\muK_{n}}\,\un_{\{n>1\}}\right)\,\phi(n)^{2}\\
&\quad -\frac{1}{2}\sqrt{\lambdaK_{N}\muK_{N+1}}\,(\phi(N)-\phi(N+1))^{2}+\frac{1}{2}\sqrt{\lambdaK_{N}\muK_{N+1}}\,\phi(N+1)^{2}\\
& \quad -\frac{1}{2}\sqrt{\lambdaK_{1}\muK_{2}}\,\phi_{1}^{2}-\frac{1}{2}\sqrt{\lambdaK_{N}\muK_{N+1}}\,\phi(N)^{2}.
\end{align*}
Since $\phi \in \Dom(\LKl)$, the functions $\phi\, \un_{\{n\le N\}}$ and $\LKl( \phi \un_{\{n\le N\}})$  converge to $\phi$, respectively $\LKl \phi$, in $\ldeux$ when $N$ tends to 
infinity. The result follows  by letting $N$ tend to infinity, since $V_{n}(K)$ is positive for $n$ large enough, and since $\lambda_{N}(K)$ and $\mu_{N}(K)$ are exponential in 
$N$ and $\phi$ decays exponentially fast by assumption.
\end{proof}

\begin{lemma}\label{borneVn}
There exists $\xi>0$ such that for all $K\in\integers$, $\inf_{n\geq 1} \Vn\geq -\xi$.
\end{lemma}
\begin{proof}
Since $(\lambdaK_{n})_n$ is an increasing sequence we have
\begin{align*}
\Vn
&=\lambdaK_{n}+\muK_{n}-\,\sqrt{\lambdaK_{n}\muK_{n+1}}-\,\sqrt{\lambdaK_{n-1}\,\muK_{n}}\,\un_{\{n>1\}}\\
& \geq \frac{\lambdaK_{n}+\muK_{n}}{2}-\,\sqrt{\lambdaK_{n}\muK_{n+1}}.
\end{align*}
It follows from the general assumptions (see Section \ref{sec:assumptions}) that there exists $\tilde{x}\geq 1$ such that all $K\in\integers$ and for all
$n\geq K\tilde{x}$ we have
\[
\frac{\lambdaK_n}{\muK_n}\leq \frac{1}{5}\quad\text{and}\quad \frac{\lambdaK_{n+1}}{\muK_n}\leq \frac{5}{4}.
\]
For all $n\geq K\tilde{x}$ we have $\Vn\geq 0$. When $n\leq K\tilde{x}$, we write
\begin{align*}
\MoveEqLeft[5] \frac{\lambdaK_{n}+\muK_{n}}{2}-\,\sqrt{\lambdaK_{n}\muK_{n+1}} \\
&= \frac{\big(\sqrt{\lambdaK_{n}}-\sqrt{\muK_{n}}\big)^2}{2}+\frac{\lambdaK_{n}}{\sqrt{\lambdaK_{n}}+\sqrt{\muK_{n+1}}}\big(\muK_n-\muK_{n+1}\big).
\end{align*}
Now observe that
\[
\muK_n-\muK_{n+1}=K\left(d\left(\frac{n}{K}\right)-d\left(\frac{n+1}{K}\right)\right)\geq -\sup_{0\,\leq\, x\,\leq \,\tilde{x}+1} d'(x).
\]
The rest of the proof is obvious.
\end{proof} 

\begin{proposition}\label{sobmodif}
Let $\delta>0$ and $\phi$ be a real normalized sequence in $\Dom(\LKl)$  decaying exponentially fast, such that
$\big\|\LKl\phi+\rho\,\phi\big\|_{\ldeux}\le \delta.$
Then
\begin{align*}
\MoveEqLeft \sum_{n=1}^{\infty}\sqrt{\lambdaK_{n}\muK_{n+1}}\,\big(\phi(n+1)-\phi(n)\big)^{2}
+\sum_{n=1}^{\infty}(1\vee\Vn)\,\big(\phi(n)\big)^{2}\\
& \le 1+\rho+\xi+\delta+\frac{1}{2}\sqrt{\lambdaK_{1}\muK_{2}}.
\end{align*}
\end{proposition}
The proof is left to the reader. It is a direct consequence of Proposition \ref{diri} and Lemma \ref{borneVn}.

\section{Spectral theory of \texorpdfstring{$\PO$}{PO}}

Recall that (cf. \eqref{le-M0})
\begin{align*}
& \big(\PO v\big)(n)=\\
\nonumber
& \sqrt{\tlamz\,\tmuz\,n\,(n+1)} \,v(n+1)+\sqrt{\tlamz\,\tmuz\,n\,(n-1)}\,v(n-1)\,\un_{\{n>1\}}\\
\nonumber
& -n\,(\tlamz+\tmuz) v(n)).
\end{align*}

\begin{theorem}\label{specdeux}
The operator $\PO$ defined on $c_{00}$  is symmetric for the scalar product of $\ldeux$.
We denote by $\PO$ its closure which is  self-adjoint and bounded above. 
The spectrum of\, $\PO$ is discrete, all eigenvalues are simple, and we have
\[
\spectre(\PO)=\big(\tmuz-\tlamz\big)\integers=-S_2.
\]
The eigenvector $v_{m}$ corresponding  to the eigenvalue  $(\tmuz-\tlamz\big)m$, where $m\in\integers$, is given (up to a multiplicative factor) by  
\begin{equation}
\label{fatigue}
v_{m}(n)=\sqrt{n}\,\left(\frac{\tmuz}{\tlamz}\right)^{\frac{n}{2}} P_{m}(n)
\end{equation} 
where $P_{m}$ is the monic orthogonal polynomial of degree $m-1$ associated with the measure $q$ on $\integers$ defined by
\begin{equation}
\label{q}
q(n)=n\left(\frac{\tmuz}{\tlamz}\right)^{n}.
\end{equation} 

\end{theorem}

\begin{proof}
It is easy to verify that  $\PO$ is a symmetric operator on $c_{00}$, which is bounded above since from $\tlamz>\tmuz$, we have
\[
\inf_{n}\big(n\,(\tlamz+\tmuz)-\sqrt{\tlamz\,\tmuz\,n\,(n+1)} - \sqrt{\tlamz\,\tmuz\,n\,(n-1)}>-\infty.
\]
It is easy to verify that  $\PO$ is closable and we denote by $\PO$ its closure. 

Since for any $m\in \integers$,  the sequence $v_{m}(n)$ defined by \eqref{fatigue} decays exponentially fast with $n$, it
is easy to verify that $v_{m}\in \Dom(\PO)$. Note also that if $m\neq
m'$, $v_{m}$ is orthogonal to $v_{m'}$ in $\ldeux$. 

By a direct computation one checks that $\PO v_{1}=(\tmuz-\tlamz)v_{1}$ (recall that $P_{1}(n)=1$).
It is left to the reader to check that 
\[
\PO v_{m}(n)=\sqrt{n}\left(\frac{\tlamz}{\tmuz}\right)^{\frac{n}{2}}\,Q_{m}(n)
\]
where $Q_{m}$ is a polynomial in $n$ in which the coefficient of $n^{m-1}$ is
\[
m (\tmuz-\tlamz)\;.
\]
To check that the $v_{m}$ are eigenvectors, we use a recursive argument. Assume that $m\ge 2$ and for $1\le k\le m-1$
\[
\PO v_{k}=k\,(\tmuz-\tlamz)\,v_{k}.
\]
We can write 
\[
\PO v_{m}=m\,(\tmuz-\tlamz)\,v_{m}+r_{m}
\]
with
\[
r_{m}=\sqrt{n}\left(\frac{\tlamz}{\tmuz}\right)^{\frac{n}{2}}R_{m}
\]
where $R_{m}$ is a polynomial in $n$ of degree at most $m-2$. 
Therefore
\[
r_{m}\in\text{Span}\big\{v_{1},\ldots, v_{m-1}\big\}\;.
\]
From our recursive assumption, the symmetry of $\PO$, and the orthogonality of the $v_{k}$ (following from the orthogonality of the $P_{k}$), we get that  for any $1\le k\le m-1$
\[
0=\langle v_{k}, \PO v_{m}\rangle_{\ldeux}=\langle v_{k},r_{m}\rangle_{\ldeux}.
\]
Therefore $r_{m}=0$. Hence $\PO v_{m}=m\,(\tmuz-\tlamz)\,v_{m}$, and we can proceed with the recursion.

We now prove that the $v_{m}$ form a basis of $\ldeux$.
Assume the contrary, namely there exists $u\in\ldeux$ of norm one such that for any $m$
\[
\sum_{n=1}^{\infty}\overline u(n)\,v_{m}(n)=0.
\]
We observe that the sequence
\[
w(n)=\frac{1}{\sqrt{n}}\left(\frac{\tlamz}{\tmuz}\right)^{\frac{n}{2}}u(n)
\]
belongs to $\ell^{2}(q)$, whith $q$ defined in \eqref{q}. Therefore our assumption on $u$ implies that $w$ is orthogonal to all the polynomials in $\ell^{2}(q)$.  Let us 
show that the set of polynomials is dense  in $\ell^{2}(q)$. It is sufficient to prove that the measure $q$ is the solution of a determinate moment problem,
see \cite[Corollary 2.50, p. 30]{Deift}. Following \cite[Proposition 1.5, p. 88]{Simon}, it is enough to prove that the moments of order $m$, denoted by
$\gamma_{m}$ of $q$, satisfy the following property: there exists $C>0$ such that, for any $m\in\integers$,
\[
\gamma_{m}=\sum_{n=1}^{\infty}n^{m+1}  \left(\frac{\tmuz}{\tlamz}\right)^{n}\le C^{m}\,m! \,.
\]
The proof is left to the reader. Therefore the set of all polynomials is dense in $\ell^{2}(q)$ implying $w=0$ and we get  a contradiction with the existence of a $u$ nonzero orthogonal 
to all the $v_{m}$ in $\ldeux$. Therefore,  the $v_{m}$ form a basis of $\ldeux$.

\smallskip We now observe that $\PO$ is bounded above. The proof is similar to that of Proposition \ref{diri} and left to the reader.
Since the $v_m$'s form a basis of $\ldeux$, for any $B>0$ we have $\text{ker}(\PO^{\dagger}-B)=\{0\}$.
Hence $\PO$ is self adjoint (see for instance \cite[Prop. 3.9, p. 43]{KS}) and the spectrum is given by 
\[
\spectre(\PO)=(\tmuz-\tlamz)\integers.
\]
This ends the proof.
\end{proof}

\section{Local maximum principle and consequences thereof}

We will state and prove a maximum/minimum principle in a form which is well suited for our purposes.
We start with a proposition giving elementary inequalities following from the order on the real line.

\begin{proposition}\label{locmax}
Assume $a>0$, $c>0$ and $b>a+c$. Let $u,w\in\real$.\newline
If $v>0$ is such that $a\, u+c\,w-b\,v\ge0$, then $v< \max\{u,w\}$.\newline
If $v<0$ is such that $a\, u+c\,w-b\,v\le0$, then $v> \min\{u,w\}$.\newline
Moreover, if $u\ge v\ge w$ are such that $a\, u+c\,w-b\,v\ge0$, then $v\le \frac{a}{b-c}\,u$.
\end{proposition}
\begin{proof}
If $v>0$ we have $0< (b-a-c)\,v\le a\,(u-v)+c\,(w-v)$ leads to a contradiction if $v\ge\max\{u,w\}$. The case $v<0$ is similar.
The last statement is trivial since $b\,v\le a\,u+c\,w\le a\,u+c\,v$.
\end{proof}

\begin{proposition}\label{principe}
Let $1<n_1<n_2$ be integers such $n_2> n_1+1$. Let $(\alpha_n)$ be a finite sequence of strictly positive real numbers defined for
$n_1-1,\ldots,n_2$.  Let $(\beta_n)$ be a finite sequence of strictly positive real numbers defined for $n_1,\ldots,n_2$. 
Let $(u_n)$ be a finite sequence of real numbers defined for $n_1-1,\ldots,n_2+1$.  Assume that, for all $n_1\leq n\leq n_2$, we
have $\beta_n> \alpha_n+\alpha_{n-1}$.
\newline If $\alpha_n u_{n+1}+\alpha_{n-1}u_{n-1}-\beta_n u_n\geq 0$, then the sequence $(u_n)$ has no positive local maxima for
$n \in\{n_1+1,\ldots,n_2-1\}$.  Moreover, if there exists some $u_n>0$ then the maximum is  attained only at the boundary, that
is, on the set $\{n_1,n_2\}$.
\newline If $\alpha_n u_{n+1}+\alpha_{n-1}u_{n-1}-\beta_n u_n\leq 0$, then the sequence $(u_n)$ has no positive local minima for
$n \in\{n_1+1,\ldots,n_2-1\}$, and if there exists some $u_n<0$ then the minimum is  attained only at the boundary, that is, on the
set $\{n_1,n_2\}$.
\end{proposition}

\begin{proof}
It follows from Proposition \ref{locmax}.
\end{proof}

\begin{proposition}\label{pouet}
Let $1<n_1<n_2$ be integers such $n_2> n_1+1$.  Let $(\alpha_n)$ be a finite sequence of strictly positive real numbers defined for 
$n_1-1,\ldots,n_2$.  Let $(\beta_n)$ be a finite sequence of strictly positive real numbers defined for $n_1,\ldots,n_2$. 
Let $(u_n)$ be a finite sequence of real numbers defined for $n_1-1,\ldots,n_2+1$.  Assume that, for all $n_1\leq n\leq n_2$, we
have $\beta_n> \alpha_n+\alpha_{n-1}$. \newline 
If $\alpha_n u_{n+1}+\alpha_{n-1}u_{n-1}-\beta_n u_n\geq 0$, $u_{n_1+1}>0$ and $u_{n_1+1}\geq u_{n_1}$, then the sequence $(u_n)$
is increasing.
\newline If $\alpha_n u_{n+1}+\alpha_{n-1}u_{n-1}-\beta_n u_n\leq 0$, $u_{n_1+1}<0$ and $u_{n_1+1}\leq u_{n_1}$, then the sequence $(u_n)$
is decreasing.\newline Finally, if $(u_n)$ is a positive sequence then there cannot be two local (positive) minima separated by a distance larger than one.
\end{proposition}

\begin{proof}
It follows recursively from Proposition \ref{locmax}.
\end{proof}

\appendix 

\section{Quasi-eigenvalues and quasi-eigenvectors of self-adjoint operators}

\begin{proposition}\label{yaspectre}
Let $A$ be a self-adjoint operator in a Hilbert space $\mathscr{H}$ with domain $\Dom(A)$. Assume there exists  $u\in \Dom(A)$ of norm $1$, $\omega\in\real$ and $\epsilon>0$ such that
\[
\big\|A\,u-\omega\,u\big\|\le\epsilon.
\]
Then 
\[
\spectre(A)\,\cap\left[\,\omega-\epsilon,\omega+\epsilon\,\right]\neq\emptyset\;.
\]
\end{proposition}
\begin{proof}
We will assume that $\omega\notin \spectre(A)$, otherwise the result is trivial. The proof is then by contradiction.
If $R_{\omega}$ denotes the resolvent of $A$ at $\omega$, we have 
\[
u=-R_{\omega}\,(A\,u-\omega\,u)
\]
hence
\[
1\le \epsilon\,\big\|R_{\omega}\big\|.
\]
The result follows from the estimate (a  direct consequence of the
spectral decomposition) 
\[
\big\|R_{\omega}\big\|\le \frac{1}{\text{d}\big(\omega,\spectre(A)\big)}
\]
where $d$ denotes the Euclidean distance on the real line. 
If $\spectre(A)\,\cap\,[\omega-\epsilon,\omega+\epsilon]=\emptyset$, since $\spectre(A)$ is closed, then
\[
\delta=\text{d}\big(\omega,\spectre(A)\big)>\epsilon
\]
and we get
\[
1\le\frac{\epsilon}{\delta}<1
\]
which is a contradiction.
\end{proof}

\begin{proposition}\label{approxvp}
Let $A$ be a self-adjoint operator in a Hilbert space $\mathscr{H}$ with domain $\Dom(A)$. Assume there exists  $u\in \Dom(A)$ of norm $1$,  $\omega\in\real$ and $\epsilon>0$ such that
\[
\big\|A\,u-\omega\,u\big\|\le\epsilon\;.
\]
Assume $A$ has discrete spectrum with eigenvalues of multiplicity one, and let $\delta>0$ denote the minimum distance between two
consecutive eigenvalues. Then if $\epsilon<\delta$ there is a $\lambda\in\spectre(A)$ with a normalized eigenvector $e$ such that $|\omega-\lambda|\le\epsilon$ and  
\[
\big\|u-e\big\|\le \frac{2\,\epsilon}{\delta-\epsilon}. 
\]
\end{proposition}

\begin{proof}
We will denote by $P_{z}$ the one-dimensional spectral projector of $A$ corresponding to $z\in\spectre(A)$.

Since $\epsilon<\delta$, using Proposition \ref{yaspectre} we conclude that there is only one eigenvalue of $A$ in $\left[\,\omega-\epsilon,\omega+\epsilon\,\right]$
and we denote by $\lambda$ this eigenvalue and by $\tilde e$ one of the corresponding eigenvectors  (they all differ only by a phase factor). 
Let 
\[
v=\omega\,u-A\,u.
\]
Since $u=P_\lambda u + (\mathrm{Id}-P_\lambda)u$, we get from the spectral decomposition 
\[
A(\mathrm{Id}-P_\lambda)\,u-\omega (\mathrm{Id}-P_\lambda)\,u=(\mathrm{Id}-P_\lambda)\,v.
\]
This implies that
\[
\| (\mathrm{Id}-P_\lambda)\,u\|\leq \frac{\epsilon}{\delta-\epsilon}.
\]
Since $P_{\lambda}\,u$ is proportional to $\tilde e$, we can write
\[
u=\beta\,\tilde e+(\mathrm{Id}-P_\lambda)\,u
\]
with $\beta\in\complex$. Since $u$ and $\tilde e$ are of norm one, and $\tilde e$ and $(\mathrm{Id}-P_\lambda)\,u$ are orthogonal, we get
\[
1=|\beta|^{2}+\big\|(\mathrm{Id}-P_\lambda)\,u\big\|^{2}
\]
which implies
\[
1\ge |\beta|\ge 1-\frac{\epsilon}{\delta-\epsilon}\;.
\]
Let $\beta=|\beta|\exp(\ic\theta)$, we define $e=\e^{\ic\,\theta}\,\tilde e$, and the result follows.
\end{proof}


\end{document}